\documentclass{article}

\usepackage{amsmath, mathrsfs, amssymb,amsthm,hyperref,tikz,tikz-3dplot,centernot,comment}
\usepackage{tikz-cd}
\usepackage{enumerate}
\hypersetup{pdfstartview={XYZ null null 1.5}}
\usepackage[all]{xy}

\theoremstyle{plain}
\newtheorem{thm}{Theorem}[section]
\newtheorem{prop}[thm]{Proposition}
\newtheorem{cor}[thm]{Corollary}
\newtheorem{lemma}[thm]{Lemma}
\theoremstyle{definition}
\newtheorem{defn}[thm]{Definition}
\newtheorem{ex}[thm]{Example}

\numberwithin{equation}{section}

\newcommand{\pd}{p^\downarrow}

\newcommand{\pu}{p^\uparrow}
\newcommand{\bv}{\bigvee}
\newcommand{\bw}{\bigwedge}
\newcommand{\cU}{\mathcal U}
\newcommand{\cV}{\mathcal V}

\newcommand{\qd}{q^\downarrow}

\newcommand{\cI}{\mathcal{I}}
\newcommand{\xd}{x^\downarrow}
\newcommand{\Frm}{\mathbf{Frm}}
\newcommand{\F}{\Gamma_{(-)}}
\newcommand{\G}{\cU_{(-)}}
\newcommand{\Pos}{\mathbf{Pos}}
\newcommand{\PosU}{\mathbf{Pos}_{\cU}}
\newcommand{\PosIU}{\mathbf{Pos}^{i}_{\cU}}

\newcommand{\LatR}{\mathbf{CLat}_R}

\newcommand{\JF}{\mathbf{JSpecF}_P}
\newcommand{\JC}{\mathbf{JComp}_P}
\newcommand{\Clos}{\mathbf{Clos}_P}
\newcommand{\JCop}{\mathbf{JComp}_{\overleftarrow{P}}}
\newcommand{\cA}{\mathcal{A}}

\begin{document}
\title{Categories of frame-completions and join-specifications}
\author{Rob Egrot}
\date{}

\maketitle
\begin{abstract}
Given a poset $P$, a join-specification $\cU$ for $P$ is a set of subsets of $P$ whose joins are all defined.
 The set $\cI_\cU$ of downsets closed under joins of sets in $\cU$ forms a complete lattice, and is, in a sense, the free $\cU$-join preserving join-completion of $P$. The main aim of this paper
 is to address two questions. First, given a join-specification $\cU$, when is $\cI_\cU$ a frame? And second,  given a poset $P$, what is the
 structure of its set of frame-generating join-specifications?

To answer the first question we provide a number of equivalent conditions, and we use these to investigate the second.
 In particular, we show that the set of frame-generating join-specifications for $P$ forms a complete lattice ordered by
 inclusion, and we describe its meet and join operations. We do the same for the set of `maximal' such
 join-specifications, for a natural definition of `maximal'.
We also define functors from these lattices, considered as categories, into a suitably defined category of frame-completions of $P$, and construct right adjoints for them.

\end{abstract}

\section{Introduction}
When passing from latices to semilattices, the concept of distributivity can be generalized in a number of ways, and previously equivalent properties diverge in the more general setting. One method is to consider a (meet) semilattice to be distributive if and only if it has the property that whenever \[y\wedge (x_1\vee x_2)\] is defined, so too is \[(y\wedge x_1)\vee (y\wedge x_2)\] and
the two are equal. Another, which we will not discuss, is due to Gr\"atzer and Schmidt \cite{GraSchm62}. This definition, proposed by Schein \cite{Sch72} and based on the work of Balbes \cite{Bal69}, is,
assuming the axiom of choice, equivalent to being embeddable into a distributive lattice via a semilattice morphism that preserves binary joins whenever they exist (see either \cite{Sch72} or \cite[theorem 2.2]{Bal69}).

This admits an obvious generalization, considered in both \cite{Sch72} and \cite{Bal69}, to arbitrary countable numbers. So, for example, for $0\leq n\leq \omega$ we can define a semilattice to be $n$-distributive
if and only if whenever $k<n$ and $y\wedge (x_1\vee\ldots \vee x_k)$ is defined, so too is $(y\wedge x_1)\vee\ldots \vee (y\wedge x_k)$ and the two are equal.
Obviously for $m<n$ we have $n$-distributivity implies $m$-distributivity, and that the converse does not hold in general is stated in \cite{Sch72}, and proved explicitly in \cite{Kea97}.
This contrasts with the lattice case, where, of course, being $3$-distributive implies being $\omega$-distributive.

Central to the relationship between these definitions and the property of being embeddable into distributive lattices via appropriately join-preserving semilattice morphisms are various separation
properties generalizing the prime ideal theorem for distributive lattices. Again, separation properties that are equivalent for lattices may not be for semilattices (see Varlet \cite{Var75}).

Alongside the development of this concept of distributivity for semilattices, a general theory of extensions and completions of posets was being established (for example \cite{Doc67, Schm72a,Schm72b}),
extending the pioneering work of MacNeille \cite{Mac37}. A notable result from this period being that every join-completion of a poset is `universal', in
the sense that it provides the object part of a left adjoint map, provided the morphisms are chosen appropriately \cite{Schm74}.

Of particular relevance here is the concept of an `ideal completion'. That is, a completion of a poset $P$ whose elements are downsets of $P$ with certain closure properties. Unfortunately,
the word `ideal' appears in the order theory literature with a variety of similar but different meanings (see \cite{Nie06} for a unifying treatment). Here we will be interested in downsets
closed under certain families of existing joins.

These ideal completions turn out to be relevant to distributivity as discussed above. For $n=\omega$ \cite[theorem 1.1]{CorHic78}, and $n<\omega$ \cite[theorem 2]{Hic78}, it turns out that a
meet-semilattice is $n$-distributive if and only if the ideal completion corresponding to downsets closed under existing joins of cardinality less than $n$ is distributive.
Though this is not stated explicitly in the statements of the cited theorems, for these values of $n$ this ideal completion will be distributive if and only if it is a frame. This is not immediately obvious, but it can be seen, for example, by appealing to our theorem \ref{T:gen}. Recent work on extending semilattices to frames using sites and coverages can be found in \cite{BallPul14}.

In \cite{HicMon84}, a concept of $\alpha$-distributivity is defined for posets and regular cardinals $\alpha$. Again, this turns out to be equivalent to the appropriate ideal completion being a frame \cite[theorem 2.7]{HicMon84}. Recently, this was partially generalized to arbitrary choices of joins to be preserved \cite[theorem 3.5]{Egr17b}. The relevant result being that the corresponding ideal completion will be a frame if and only if the ideal `closure' of an arbitrary subset can be constructed using a certain recursive process. In that paper this is then used to prove some results about embedding posets in powerset algebras.

The process defined in \cite{Egr17b} is actually unnecessarily complicated, and the `recursive' process always terminates in one step. Section \ref{S:ideals} of this paper is devoted to proving this, and using this observation to fully extend \cite[theorem 2.7]{HicMon84} to the more general setting.

The important idea is this. In \cite{CorHic78,Hic78,HicMon84}, the crucial fact is that a relevant ideal can be constructed from an arbitrary down-closed set $S$ purely by an iterated process of closing appropriately with respect to certain existing joins, and, in the case of posets, closing downwards. When certain conditions are met, this closing process terminates after the first application, and, in the poset case does not even require closing downwards. This turns out to be equivalent to the lattice of ideals being a frame. A similar idea occurs in the context of \emph{partial frames} as \cite[lemma 3.3]{FriSch16}.

When we allow more freedom in the choice of which joins our ideals must preserve (\cite{Egr17b} uses the term \emph{join-specification}), this is no longer true. However, the set of join-specifications for a poset $P$ is related to the set of join-completions of $P$ via a Galois connection (see proposition \ref{P:adj}). Thus every join-specification has a `closure', obtained taking its image under the composition of the adjoint maps. It turns out that, for these `closed' (AKA \emph{maximal}) join-specifications, the closing process for ideals requiring only a single step is again equivalent to the ideal lattice  being a frame.  The details of this can be found in section \ref{S:ideals}.

Being sets of sets, the join-specifications for a poset can be ordered by inclusion. We call a join-specification \emph{frame-generating} if the corresponding ideal lattice is a frame.  In sections \ref{S:cat} and \ref{S:comp} we consider the structure of the posets of frame-generating, and maximal frame-generating, join-specifications for a poset $P$. In particular it turns out that these are both complete lattices, with joins being given by unions in the former case, and meets being given by intersections in the latter. We shall see that the process of constructing the corresponding ideal completion from a frame-generating join-specification is functorial, and we will define a right adjoint.

The concept of a join-specification is strongly related to that of a \emph{subset selection}. Subset selections have been used, particularly by Ern\'e, to study generalizations of continuous lattices, whose study was initiated by Scott \cite{Sco72} in the early Seventies (i.e. around the same time as the other ideas discussed here).

In particular, subset selections are used to define $\mathcal Z$-continuous posets via a suitable $\mathcal Z$-\emph{below} relation. The theory of $\mathcal{Z}$-continuous posets and $\mathcal{Z}$-below relations has received considerable attention, both from an order theoretic perspective (see e.g. \cite{WWT78,Nov82,BanErn83,Ern86,ErnZha01}), and also in relation to topology (broadly construed) \cite{Ern99}.

Ideal completions are again important in the context of $\mathcal{Z}$-continuity, though here the focus has largely been on subset selections such that a poset being $\mathcal Z$-continuous corresponds to the so called $\mathcal Z$-join ideal completion being completely distributive. This generalizes the characterization of complete distributivity of a lattice discovered by Raney in the early Fifties \cite{Ran53}. In this vein see also the constructive approach to complete distributivity developed in \cite{FawWoo90} and its sequels.

Being primarily interested in ideal-completions that are frames, rather than the stronger property of complete distributivity, the work here is largely distinct from the work on $\mathcal{Z}$-continuous posets mentioned above, and we will mostly use the notation of \cite{Egr17b}. In section \ref{S:global} however we discuss some categorical questions arising naturally from our results on frames of ideals. These have analogues in the setting of $\mathcal{Z}$-continuity, and so we return to subset selections etc. to make appropriate comparisons.

\section{Extensions of posets}\label{S:ext}
 In this section we provide important definitions and results for convenience. We use notation based on, but not completely identical to, that used in \cite{Egr17b}. We assume familiarity with the basics of category theory and order theory. Readers in need of an accessible introduction to either of these topics are directed to \cite{Lein14} and \cite{DavPri02} respectively. We begin with a brief summary of our relevant notational idiosyncrasies.

\begin{itemize}
\item For $S\subseteq P$, $S^\downarrow$ denotes $\{p\in P:p\leq s$ for some $s\in S\}$. We make a similar definition for $S^\uparrow$
\item For $p\in P$, $\pd$ denotes $\{p\}^\downarrow$ (where this is unambiguous),  and we make a similar definition for $\pu$.
\item If $f$ is a map with domain $P$ we write $f[S]$ to denote $\{f(p):p\in S\}$.
\item If $\alpha$ is a cardinal then we write $\hat{\alpha}$ for the smallest regular cardinal greater than or equal to $\alpha$.
\item Indexing sets are assumed to be non-empty.
\end{itemize}

\begin{defn}\label{D:abdist}
Given cardinals $\alpha$ and $\beta$, a lattice $L$ is \textup{\textbf{$(\alpha,\beta)$-distributive}} if whenever $I$ and $J_i$ are indexing sets with $|I|< \alpha$ and $|J_i|<\beta$ for all $i\in I$, and $x_{ij}\in L$ for every $i\in I$ and $j\in J_i$, then if $\bw_I\bv_{J_i} x_{ij}$ exists, and if $\bw_I x_{if(i)}$ exists for every $f:I\to\bigcup_I J_i$ with $f(i)\in J_i$ for all $i\in I$, then $\bv_{f} \bw_I x_{if(i)}$ exists and we have
\[\bw_I\bv_{J_i} x_{ij}=\bv_{f} \bw_I x_{if(i)}.\]
\begin{itemize}
\item If $L$ is $(\alpha,\beta)$-distributive for all $\beta$ then we say $L$ is $(\alpha,\infty)$-distributive, and we define $(\infty,\beta)$-distributive similarly.
\item If $L$ is $(\alpha,\beta)$-distributive for all $\alpha$ and $\beta$ then we say that $L$ is $(\infty,\infty)$-distributive (or usually just completely distributive).
\item There is a dual definition to be made by reversing the orders of $\bv$ and $\bw$. We will not make this explicit, but we note that when $L$ is complete, $(\infty,\infty)$-distributivity and its dual are equivalent \cite[theorem 1]{Ran52}.
\end{itemize}
\end{defn}

\begin{defn}
A \textup{\textbf{frame}} is a complete lattice $L$ with the property that, given $x\in L$ and $Y\subseteq L$, we have $x\wedge \bv Y = \bv_{y\in Y} (x\wedge y)$. I.e. A complete lattice $L$ is a frame if and only if it is $(\omega,\infty)$-distributive.

If $L_1$ and $L_2$ are frames, then a map $f:L_1\to L_2$ is a \textup{\textbf{frame morphism}} if it is a completely join-preserving lattice homomorphism.
\end{defn}

\begin{defn}
If $X$, $Y_1$ and $Y_2$ are sets, and if $e_1:X \to Y_1$ and $e_2:X\to Y_2$ are functions, then we say a function $f:Y_1\to Y_2$ \textup{\textbf{fixes}} $X$ with respect to $e_1$ and $e_2$ if the diagram in figure \ref{F:fix} commutes. Often, $e_1$ and $e_2$ are easily understood from context, and in these cases we will abuse notation by just saying, e.g. ``$f$ fixes $X$", without reference to $e_1$ and $e_2$.
\end{defn}

\begin{figure}
    \centerline{
    \xymatrix{
    X\ar[r]^{e_1}\ar[dr]_{e_2} & Y_1\ar[d]^f \\
    & Y_2
    }}
    \caption{}
    \label{F:fix}
\end{figure}

\begin{defn}
If $P$ is a poset, an \textup{\textbf{extension}} of $P$ is a poset $Q$ and an order embedding $\phi: P\to Q$. We
 sometimes abuse our notation by suppressing $\phi$ and just referring to $Q$ as a poset extension. If $Q$ is a
 complete lattice we say $\phi: P\to Q$ (or just $Q$) is a \textup{\textbf{completion}} of $P$. If for all $q\in Q$ we
 have $q=\bv (\phi[P]\cap \qd)$ then we say $\phi: P\to Q$ is a \textup{\textbf{join-extension}} of $P$ and that
 $\phi[P]$ is \textup{\textbf{join-dense}} in $Q$. If $\phi:P\to Q$ is a join-extension and $Q$ is a complete lattice then then we say $\phi:P\to Q$ is a \textup{\textbf{join-completion}}, and that $\phi[P]$ \textup{\textbf{join-generates}} $Q$.

We can define notions of meet-extensions etc. dually.
\end{defn}

There is a well known correspondence between join-completions of $P$ and standard closure operators on $P$.

\begin{defn}
Let $P$ be a poset. Then a \textup{\textbf{closure operator}} on $P$ is a map $\Gamma:\wp(P)\to \wp(P)$ that is extensive, monotone and idempotent. A closure operator $\Gamma$ is \textup{\textbf{standard}} if $\Gamma(\{p\}) = \pd$ for all $p\in P$. A set $S\subseteq P$ is \textup{\textbf{$\Gamma$-closed}} if $\Gamma(S) = S$.
\end{defn}

Given $P$, the standard closure operators on $P$ form a complete lattice when ordered by pointwise inclusion (i.e.
 $\Gamma_1\leq \Gamma_2 \iff \Gamma_1(S)\subseteq \Gamma_2(S)$ for all $S\subseteq P$).
In this lattice, meets are defined using $(\bw_I \Gamma_i)(S)=\bigcap_I \Gamma_i(S)$ for all $S\subseteq P$.

Standard closure operators and join-completions are essentially the same thing, as proposition \ref{P:closure} below makes precise.
First though we present a basic result about closure operators in a form that will be useful to us later.

\begin{prop}\label{P:arrow}
Let $\Gamma$ be a standard closure operator on a poset $P$, and let $L_\Gamma$ be the complete lattice of $\Gamma$-closed subsets of $P$. Suppose $\Gamma'$ is another standard closure operator, and that $\Gamma' \leq \Gamma$. Let $\eta:P\to L_\Gamma$ and $\eta':P\to L_{\Gamma'}$ be the canonical embeddings $p\mapsto \pd$. Then there is a unique completely join-preserving map $\phi: L_{\Gamma'}\to L_\Gamma$ fixing $P$.
Moreover, $L_\Gamma$ is a frame if and only if $\phi$ preserves finite meets for all choices of $\Gamma'\leq \Gamma$.
\end{prop}
\begin{proof}
Define $\phi$ by $\phi(C) = \Gamma(C)$. Then commutativity of the diagram is automatic. Let $I$ be an indexing set, and let $C_i\in L_{\Gamma'}$ for all $i\in I$. Then by basic properties of closure operators we have
\[\Gamma(\bv_I C_i) = \Gamma(\Gamma'(\bigcup_I C_i)) = \Gamma(\bigcup_I C_i) = \Gamma(\bigcup_I \Gamma(C_i)) = \bv_I \Gamma(C_i),\]
so $\phi$ is completely join-preserving. Moreover, this is clearly the only possible choice of map that is both completely join-preserving and fixes $P$.

Finally, suppose $I= \{1,\ldots,n\}$ is finite, that $L_\Gamma$ is a frame, and that $x\in \bigcap_I \Gamma(C_i)$. Then $\xd \subseteq \bigcap_I \Gamma(C_i)$, or, equivalently, $\xd = \xd \cap (\bv C_1 \cap \ldots \cap \bv C_n)$, where $\bv C_i$ stands for the join in $L_\Gamma$ of $\{c^\downarrow : c\in C\}$.
Now, using basic properties of closure operators and the fact that $L_\Gamma$ is a frame, and noting that $C_i$ is down-closed for all $i\in I$, we have
\begin{align*}
\xd \cap (\bv C_1 \cap \ldots \cap \bv C_n) &= \bv_{c_1\in C_1} \ldots \bv_{c_n\in C_n}(\xd \cap c_1^\downarrow\cap \ldots \cap c_n^\downarrow) \\
&= \Gamma(\bigcup_{c_1\in C_1}(\ldots \Gamma(\bigcup_{c_n\in C_n}(\xd\cap c_1^\downarrow\cap \ldots \cap c_n^\downarrow))\ldots))\\
&= \Gamma(\bigcup_{c_1\in C_1}\ldots \bigcup_{c_n\in C_n}(\xd\cap c_1^\downarrow\cap \ldots \cap c_n^\downarrow))\\
&= \Gamma(\xd\cap \bigcap_I C_i).
\end{align*}
So $x\in\Gamma(\bigcap_I C_i)$, and thus $\bigcap_I \Gamma(C_i) \subseteq \Gamma(\bigcap_I C_i)$, and so we have equality as the other inclusion always holds.

Conversely, if $\phi$ is a frame morphism for all choices of $\Gamma'\leq \Gamma$, then $L_\Gamma$ is a completely join-preserving homomorphic image of the complete lattice of all downsets of $P$, denoted by $\cA(P)$, which is completely distributive. It follows then that $L_\Gamma$ is a frame.
\end{proof}

\begin{defn}\label{D:JC}
Let $\JC$ be the category whose objects are the join-completions of $P$, and whose arrows are the (commuting triangles induced by) completely join-preserving maps fixing $P$.
\end{defn}

\begin{prop}\label{P:closure}
 Let $\Clos$ be the complete lattice of standard closure operators on $P$, ordered by pointwise inclusion, considered as a category.

Let $\delta$ be the map that takes a standard closure operator $\Gamma$ to the embedding $\eta:P\to L_\Gamma$ defined by $\eta(p)=\pd$, where $L_\Gamma$ is the lattice of $\Gamma$-closed subsets of $P$.
Let $\gamma$ be the map that takes a join-completion $e:P\to L$ to the standard closure operator induced by taking its closed sets to be all those of form $e^{-1}(\xd)$ for $x\in L$.

Then $\delta$ and $\gamma$ extend to functors, and induce an equivalence between $\JC$ and $\Clos$.
\end{prop}
\begin{proof}
A direct proof using proposition \ref{P:arrow} is not difficult.
\end{proof}

The result of proposition \ref{P:closure} is often stated in dual form. For example, see the introduction to \cite{Schm72a}. We return to this duality in section \ref{S:alt}.

The following result is well known and straightforward to prove.
\begin{prop}\label{P:ext}
If $\phi:P\to J$ is a join-extension then whenever $S\subseteq P$ and $\bw S$ is defined in $P$ we have $\bw\phi[S]=\phi(\bw S)$ in $J$.
\end{prop}

\begin{defn}\label{D:jspec}
Let $P$ be a poset. Let $\cU$ be a subset of $\wp(P)$. Then $\cU$ is a \textup{\textbf{join-specification}} (of $P$) if it satisfies the following conditions:
\begin{enumerate}[(1)]
\item $\bv S$ exists in $P$ for all $S\in \cU$, and
\item $\{p\}\in \cU$ for all $p\in P$.
\end{enumerate}
\end{defn}
This is a slight departure from \cite{Egr17b}, where it is required that join-specifications do not include $\emptyset$.

\begin{defn}\label{D:ideal} Given a join-specification $\cU$ we define a $\cU$-\textup{\textbf{ideal}}  of $P$ to be a subset of $P$ that is closed downwards and under joins from $\cU$.
\end{defn}
There is a slight subtlety about whether $\emptyset$ is a $\cU$-ideal or not. The answer depends on whether $P$ has a bottom, and also whether $\emptyset\in\cU$. The table in figure \ref{F:empty} explains the situation in each case. Note that it is not possible for $P$ not to have a bottom and for $\emptyset\in\cU$, because this would contradict the first condition in the definition of a join-specification.

\begin{figure}[htbp]
\centering
\begin{tabular}{|l|l|l|}\hline
 & $\emptyset\in \cU$ & $\emptyset\not\in\cU$ \\ \hline
$P$ has bottom & No & Yes \\ \hline
$P$ has no bottom & Can't happen & Yes \\ \hline
\end{tabular}
\caption{Is $\emptyset$ a $\cU$-ideal?
}
\label{F:empty}
\end{figure}

\begin{defn}[$\eta:P\to\cI_\cU$]\label{D:eta}
Given a poset $P$ and a join-specification $\cU$ for $P$, we define $\cI_\cU$ to be the complete lattice of $\cU$-ideals of $P$ (ordered by inclusion). Define $\eta$ by $\eta(p)=\pd$ for all $p\in P$. We will sometimes mildly abuse our notation by using the common label `$\eta$' for maps that are technically different, but they will all be defined by $p\mapsto \pd$ so this should not cause confusion.
\end{defn}

\begin{defn}\label{D:pres}
Let $P$ be a poset, let $\cU$ be a join-specification for $P$, let $Q$ be another poset, and let $f:P\to Q$ be a map (not necessarily monotone). Then we say $f$ is $\cU$\textup{\textbf{-preserving}} if for all $S\in \cU$ we have $f(\bv
 S)=\bv f[S]$. We may also say $f$ \textup{\textbf{preserves joins from}} $\cU$, or that $f$ \textup{\textbf{preserves
 $\cU$-joins}}. If $f$ preserves all existing joins we say it is
 \textup{\textbf{completely join-preserving}}. When a $\cU$-preserving map is also monotone we say that it is a
 $\cU$\textup{\textbf{-morphism}}. Since completely join-preserving maps are necessarily monotone, a completely join-preserving map is a $\cU$-morphism for all $\cU$.
\end{defn}

\begin{prop}\label{P:pres}
$\eta:P\to \cI_\cU$ is a completely meet-preserving $\cU$-embedding, and also defines a join-completion.
\end{prop}
\begin{proof}
This is standard. A detailed argument can be found, for example, in \cite[proposition 2.9]{Egr17b}.
\end{proof}

\begin{defn}\label{D:rad}
Given a join-specification $\cU$ we define the \textup{\textbf{radius}} of $\cU$ to be the smallest cardinal $\sigma$ such that $\sigma> |S|$ for all $S\in\cU$.
\end{defn}

\begin{defn}[$\Gamma_\cU$]\label{D:G}
Given a join-specification $\cU$ with radius $\sigma$ and $S\subseteq P$ we define the following subsets of $P$ using transfinite recursion.
\begin{itemize}
\item $\Gamma_\cU^0(S) = S^\downarrow$.
\item If $\alpha+1$ is a successor ordinal then \[\Gamma_\cU^{\alpha+1}(S)=\{\bv T : T\in\cU\text{ and }T\subseteq \Gamma_\cU^\alpha(S)\}^\downarrow.\]
\item If $\lambda$ is a limit ordinal then $\Gamma_\cU^\lambda(S)=\bigcup_{\beta<\lambda}\Gamma_\cU^\beta(S)$.
\end{itemize}
We define $\Gamma_\cU:\wp(P)\to\wp(P)$ by $\Gamma_\cU(S)=\Gamma_\cU^{\hat\sigma}(S)$ for all $S\in\wp(P)$ (recalling our convention that $\hat\sigma$ is the smallest regular cardinal greater than or equal to $\sigma$).
\end{defn}

\begin{prop}\label{P:trans}
If $\cU$ is a join-specification for $P$ then $\Gamma_\cU$ is the standard closure operator taking $S\subseteq P$ to the smallest $\cU$-ideal containing $S$.
\end{prop}
\begin{proof}
The straightforward approach is sufficient. An explicit argument can be found as \cite[proposition 2.8]{Egr17b}. Only minor adjustments are needed to accommodate the different treatment of the empty set in definition \ref{D:ideal}.
\end{proof}

\begin{cor}\label{C:arrow}
Let $\cI_\cU$ be a frame, and for a finite indexing set $J$ let $S_j\subseteq P$ for all $j\in J$. Then $\Gamma_\cU(\bigcap_J S_j^\downarrow) = \bigcap_J \Gamma_\cU(S_j)$.
\end{cor}
\begin{proof}
This follows from propositions \ref{P:trans} and \ref{P:arrow} by setting $\Gamma = \Gamma_\cU$ and $\Gamma'= \Gamma_{B_P}$, where $B_P$ is the minimal join-specification (see definition \ref{D:B} below).
\end{proof}
The restriction to finite $J$ in corollary \ref{C:arrow} is necessary, as we shall see in example \ref{E:strength} later.

Note that given $P$ and $\cU$, the $\cU$-ideal completion $\eta:P\to \cI_\cU$ is universal in the following sense.

\begin{prop}\label{P:univ}
Let $P$ be a poset, let $\cU$ be a join-specification for $P$, let $L$ be a complete lattice, and let $e:P\to L$ be a $\cU$-embedding. Then there is a unique completely join-preserving map $h:\cI_\cU\to L$ that fixes $P$ with respect to $e$ and the canonical embedding $\eta:P\to\cI_\cU$.
\end{prop}
\begin{proof}
That $h$ must be unique if it exists follows from the condition that it fixes $P$, and the fact that $\cI_\cU$ is a
 join-completion. Existence of $h$ will follow from proposition \ref{P:arrow} and the correspondence between
 join-completions of $P$ and lattices of closed sets of closure operators on $P$ from proposition \ref{P:closure}.

To see this, let
 $L'$ be the subposet of $L$ join-generated by $e[P]$. then $L'$ is not necessarily a
 sublattice of $L$, but the inclusion of $L'$ into $L$ is completely join-preserving. Moreover, $e:P\to L'$ is a
 join-completion of $P$. As $e$ is $\cU$-preserving, we can think of $L'$ as being a complete lattice of
 $\cU$-ideals, while $\cI_\cU$ is the
 complete lattice of \emph{all} $\cU$-ideals. Thus, if $\Gamma_\eta$ and $\Gamma_e$ are the closure operators associated with
 $\eta$ and $e$ respectively, we have $\Gamma_\eta\leq\Gamma_e$, and so proposition \ref{P:arrow} applies, providing a
 completely join-preserving map $h:\cI_\cU\to L'$ fixing $P$. Composing this with the inclusion of $L'$ into $L$ gives the
 required map.
\end{proof}

One may wonder under what circumstances the universal property from proposition \ref{P:univ} applies if we restrict to frames and frame morhpisms. We return to this question in section \ref{S:global}.

Standard closure operators also define join-specifications.

\begin{defn}[$\cU_\Gamma$]\label{D:UG}
If $\Gamma$ is a standard closure operator on $P$ then we define $\cU_\Gamma$ to be the join-specification on $P$ defined by
\[S\in \cU_\Gamma\iff \bv S \text { exists and } \bv S\in\Gamma(S)\]
\end{defn}

In other words, $\cU_\Gamma$ contains every set whose join is preserved by the canonical embedding $\eta:P\to L_\Gamma$. The set of join-specifications for a poset $P$ is a complete lattice when ordered by inclusion. It's easy to see that join and meet in this lattice are just union and intersection.
  Viewing the complete lattices of join-specifications and standard closure operators for a given poset $P$ as categories, we can define functors $\F$ and $\G$ by $\F:\cU\mapsto \Gamma_\cU$ and $\G:\Gamma\mapsto \cU_\Gamma$, and prove the following result.

\begin{prop}\label{P:adj}
With $\F$ and $\G$ as above:
\begin{enumerate}[(1)]
\item $\F\dashv \G$.
\item Neither $\F\circ\G$ nor $\G\circ\F$ is naturally isomorphic to the corresponding identity functor.
\item $\F\circ\G\circ\F = \F$ and $\G\circ\F\circ\G = \G$.
\end{enumerate}
\end{prop}
\begin{proof}
As the categories in question are both posets, part 1 is equivalent to saying that $\F$ and $\G$ form a Galois connection, and part 2 is equivalent to saying that there are join-specifications $\cU$ with $\cU_{\Gamma_\cU}\neq \cU$ and $\Gamma_{\cU_\Gamma}\neq\Gamma$. A proof for part 1, and examples for part 2, can be found as \cite[proposition 2.12]{Egr17b} and \cite[examples 2.14 and 2.15]{Egr17b} respectively. Part 3 follows from part 1 and the observation that this is just a Galois connection.
\end{proof}

\begin{defn}[$\cU^+$]\label{D:plus}
If $\cU$ is a join-specification then we define
\[\cU^+ = \cU_{\Gamma_\cU} = \G\circ\F (\cU).\]
\end{defn}

\begin{defn}\label{D:max}
We say a join-specification $\cU$ is \textup{\textbf{maximal}} if $\cU=\cU^+$. I.e. if it is in $\mathbf{Fix}(\G\circ\F)$.
\end{defn}

\begin{lemma}\label{L:capmax}
If $I$ is an indexing set and $\cU_i$ is a maximal join-specification for each $i\in I$, then $\bigcap_I \cU_i$ is also maximal.
\end{lemma}
\begin{proof}
Let $S\subseteq P$. Then $\Gamma_{\bigcap_I \cU_i}(S)\subseteq \bigcap_I \Gamma_{\cU_i}(S)$. Thus $\Gamma_{\bigcap_I \cU_i}\subseteq \bw_I \Gamma_{\cU_i}$, by definition of meet in the lattice of standard closure operators on $P$. So, applying the functor $\G$ from proposition \ref{P:adj} gives us

\[(\bigcap_I \cU_i)^+\subseteq \cU_{\bw_I \Gamma_{\cU_i}}= \bigcap_I \cU_i^+= \bigcap_I\cU_i,\]

where the first equality holds as right adjoints preserve limits, and the second equality holds by the maximality of the $\cU_i$. Thus $(\bigcap_I \cU_i)^+= \bigcap_I\cU_i$, as the other inclusion is automatic, and so $\bigcap_I\cU_i$ is maximal as required.

\end{proof}

It will be useful to give names to the `minimal' join-specifications.

\begin{defn}[$B_P$, $B_P^+$]\label{D:B}
Given a poset $P$ define $B_P$ to be the set of all singleton subsets of $P$. Define $B_P^+$ to be the maximal join-specification generated by $B_P$ (this is just going to be the set of all subsets of $P$ that contain their maximum).
\end{defn}

Example \ref{E:strict} later demonstrates that inclusion $(\bigcap_I \cU_i)^+\subseteq \cU_{\bw_I \Gamma_{\cU_i}}$ in the proof of lemma \ref{L:capmax} can be strict, even with some additional restrictions on the join-specifications involved.

We can consider an alternative `closing' process on $\wp(P)$, as is done in \cite{Egr17b}. This gives rise to definition \ref{D:U}, which we will revise in definition \ref{D:Unew}.

\begin{defn}[temporary]\label{D:U}
Given a join-specification $\cU$ with radius $\sigma$ and $S\subseteq P$ we define the following subsets of $P$ using transfinite recursion.
\begin{itemize}
\item $\Upsilon_\cU^0(S) = S^\downarrow$.
\item If $\alpha+1$ is a successor ordinal then \[\Upsilon_\cU^{\alpha+1}(S)=\{\bv T : T\in\cU^+\text{] and T}\subseteq \Upsilon_\cU^\alpha(S)\}.\]
\item If $\lambda$ is a limit ordinal then $\Upsilon_\cU^\lambda(S)=\bigcup_{\beta<\lambda}\Upsilon_\cU^\beta(S)$.
\end{itemize}
We define $\Upsilon_\cU:\wp(P)\to\wp(P)$ by $\Upsilon_\cU(S)=\Upsilon_\cU^{\hat\sigma}(S)$ for all $S\in\wp(P)$.
\end{defn}

This process fails to be idempotent in general, so is not usually a closure operator. However, in the special case where $\cI_\cU$ is a frame, it coincides with $\Gamma_\cU$, as shown in \cite[theorem 3.5]{Egr17b}, which we state now without proof.

\begin{thm}\label{T:old}
The following are equivalent:
\begin{enumerate}[(1)]
\item $\Gamma_\cU(S)=\Upsilon_\cU(S)$ for all $S\in\wp(S)$.
\item $\cI_\cU$ is a frame.
\end{enumerate}
\end{thm}

Further thought reveals that the construction of $\Upsilon_\cU$ is unnecessarily complicated, and that it is sufficient for our purpose to terminate the process at $\Upsilon_\cU^1$ (see lemma \ref{L:equal}). Using this observation, we will present an improved version of the theorem above as theorem \ref{T:gen}.

\section{Generating frames with join-specifications}\label{S:ideals}
This section is mainly concerned with presenting an improved version of \cite[theorem 3.5]{Egr17b}. We use several results from \cite{Egr17b}, though in many cases we can simplify their proofs. Where this is possible we provide new versions. Let $P$ be a poset and let $\cU$ be a join-specification for $P$. Let $\Gamma_\cU$ and $\Upsilon_\cU$ be as in section \ref{S:ext}.

\begin{lemma}\label{L:fix}
Let $\cU$ be a join-specification for $P$. Then the following hold for all $S\in\wp(P)$:
\begin{enumerate}[(1)]
\item If $p\geq s$ for all $s\in S$, then $p\in\Gamma_\cU(S) \implies p=\bv S$.
\item If $\bv S\in\Gamma_\cU(S)$ then $S\in \cU^+$.
\end{enumerate}
\end{lemma}
\begin{proof}
This is a minor variation of \cite[lemma 2.18]{Egr17b}.
\end{proof}

\begin{lemma}\label{L:3.3}
$\Upsilon_\cU(S)\subseteq \Gamma_{\cU}(S)$ for all $S\subseteq P$.
\end{lemma}
\begin{proof}
This is \cite[lemma 3.3]{Egr17b}. The proof here is essentially the same, but so short that we repeat it for convenience.
By proposition \ref{P:adj} we have $\Gamma_\cU = \Gamma_{\cU^+}$, and by definition of $\Upsilon_\cU$ we have $\Upsilon_\cU(S)\subseteq \Gamma_{\cU^+}(S)$ for all $S\subseteq P$.
\end{proof}

\begin{lemma}\label{L:cap}
If $x\in P$ and $S\subseteq P$, then $x\in \Upsilon_\cU(S)\iff x\in \Upsilon_\cU(\xd\cap S^\downarrow)$.
\end{lemma}
\begin{proof}
The implication from right to left is obvious. For the other implication we suppose that $x\in \Upsilon_\cU^\alpha(S)$
 and proceed by induction on $\alpha$. In the base case, $\alpha=0$, and the result is trivial. The inductive step
 where $\alpha$ is a limit ordinal is also trivial, so we assume that $\alpha=\beta+1$ for some $\beta$. Since $x\in
 \Upsilon_\cU^\alpha (S)$ it follows from the definition of $\Upsilon_\cU$ that $x=\bv T$ for some $T\subseteq
 \Upsilon_\cU^\beta (S)$ with $T\in \cU^+$. By the inductive hypothesis, for all $t\in T$ we have $t\in \Upsilon_\cU(t^\downarrow\cap
 S^\downarrow)$, and therefore $t\in \Upsilon_\cU(\xd\cap S^\downarrow)$. It follows that $x\in
 \Upsilon_\cU(\xd\cap S^\downarrow)$ as required.
\end{proof}

\begin{lemma}\label{L:equal}
If $S\subseteq P$, then $\Upsilon_\cU(S)=\Upsilon_\cU^1(S)$.
\end{lemma}
\begin{proof}
We need only show that $\Upsilon_\cU(S)\subseteq \Upsilon_\cU^1(S)$ as the other inclusion is automatic. Let $x\in
 \Upsilon_\cU(S)$. Then by lemmas \ref{L:cap} and \ref{L:3.3} we know $x\in \Upsilon_\cU(\xd\cap S^\downarrow)\subseteq
 \Gamma_\cU(\xd\cap S^\downarrow)$, and so $x=\bv (\xd\cap S^\downarrow)$ by lemma \ref{L:fix}(1). It follows that $\xd\cap
 S^\downarrow\in \cU^+$ by lemma \ref{L:fix}(2), and thus $x\in \Upsilon_\cU^1(S)$ as required.
\end{proof}

In light of lemma \ref{L:equal} we can override definition \ref{D:U}.

\begin{defn}\label{D:Unew}
\[\Upsilon_\cU(S) = \{\bv T : T\in \cU^+ \text{ and } T\subseteq S^\downarrow\} \text{ for all } S\subseteq P.\]
\end{defn}

\begin{lemma}\label{L:down}
Let $\cU$ be a join-specification for $P$, and let $S\subseteq P$. Then $\Upsilon_\cU(S)$ is the smallest $\cU$-ideal containing $S$ if and only if it is down-closed.
\end{lemma}
\begin{proof}
By definition, $\cU$-ideals must be down-closed. For the converse, we prove the contrapositive. Suppose $S\in \cU$ and that $\Upsilon_\cU(S)$ is not the smallest $\cU$-ideal containing $S$. By lemma \ref{L:3.3}, this is equivalent to saying that $\Upsilon_\cU(S)\subset\Gamma_\cU(S)$. Recall definition \ref{D:G}, and let $\alpha$ be the smallest
 ordinal with $\Upsilon_\cU(S)\subset \Gamma_\cU^{\alpha}(S)$. Let $p\in \Gamma_\cU^{\alpha}(S)\setminus
 \Upsilon_\cU(S)$. Observing that $\alpha$ must be a successor ordinal, there must be $T\subseteq
 \Gamma_\cU^{\alpha-1}(S)\subseteq\Upsilon_\cU(S)$ with $T\in\cU$ and $p\leq\bv T$. By lemma \ref{L:equal} we must have $\bv
 T\in \Upsilon_\cU(S)$, and thus $\Upsilon_\cU(S)$ is not down-closed.
\end{proof}

Definition \ref{D:des} below generalizes \cite[definition 2.2]{HicMon84}, where it is used as the basis of a definition of distributivity.
\begin{defn}\label{D:des}
If $\cU$ is a join-specification for $P$, then $P$ has the $\cU$-\textup{\textbf{descent property}} if, for every down-closed $S\in \cU$ and $p\leq \bv S$, there is $T\in\cU^+$ with $T\subseteq S^\downarrow$ and $p=\bv T$.
\end{defn}

\begin{thm}\label{T:gen}
Given a join-specification $\cU$ for $P$, the following are equivalent:
\begin{enumerate}[(1)]
\item $\cI_\cU$ is a frame.\label{f}
\item For all join-specifications $\cV$ such that $\Gamma_\cV=\Gamma_\cU$, and for all $p\in P$ and $S\in\cV$, if $p\leq \bv S$, then $\pd\cap \Gamma_\cU(S) = \Gamma_\cU(\pd\cap S^\downarrow)$. \label{T}
\item For all $p\in P$ and $S\in\cU^+$, if $p\leq \bv S$, then $\pd\cap \Gamma_\cU(S) = \Gamma_\cU(\pd\cap S^\downarrow)$.\label{T'}
\item For all $p\in P$ and $S\in\cU$, if $p\leq \bv S$, then
$\pd\cap \Gamma_\cU(S) = \Gamma_\cU(\pd\cap S^\downarrow)$. \label{T''}
\item $\Upsilon_\cU(S)$ is down-closed for all $S\in\cU$.\label{D}
\item For all $S\in \cU$, the smallest $\cU$-ideal containing $S$ is $\Upsilon_\cU(S)$. \label{U}
\item For all $S\subseteq P$, the smallest $\cU$-ideal containing $S$ is $\Upsilon_\cU(S)$.\label{i}
\item $\cI_\cU$ is $(\omega,\sigma)$-distributive (where $\sigma$ is the radius of $\cU$).\label{d}
\item Whenever $\cU'$ is a join-specification with $\cU'\subseteq \cU$ there is a unique frame morphism $\phi:\cI_{\cU'}\to \cI_{\cU}$ that fixes $P$ with respect to the canonical embeddings.
\label{h}
\item $P$ has the $\cU$-descent property. \label{des}
\end{enumerate}
\end{thm}
\begin{proof}
That (\ref{f}) and (\ref{i}) are equivalent is \cite[theorem 3.5]{Egr17b}, that (\ref{d})$\implies$(\ref{f}) is \cite[corollary 3.7]{Egr17b}, and that (\ref{T''})$\implies$(\ref{f}) is \cite[corollary 3.6]{Egr17b}. We could complete the proof by filling in the missing implications, but to avoid relying on results from \cite{Egr17b} whose proofs can be simplified, we provide an alternative, self-contained proof below. This theorem can be considered to supersede the previous results.

We prove the equivalence of (1)-(7) in a circle, and after this we deal with the remaining three.
\begin{enumerate}[(1)]
\item[](\ref{f})$\implies$(\ref{T}): Noting that $\pd\cap \Gamma_\cU(T) = \Gamma_\cU(\pd\cap T^\downarrow)$ is another way of writing $\eta(p)\wedge \bv \eta[T] = \bv_T(\eta(p)\wedge \eta(t))$,  if (\ref{T}) fails then $\cI_\cU$ is clearly not a frame.

\item[] (\ref{T})$\implies$(\ref{T'}): This follows immediately from the fact that $\Gamma_{\cU^+}=\Gamma_\cU$, which in turn follows from proposition \ref{P:adj}.

\item[] (\ref{T'})$\implies$(\ref{T''}): Trivial.

\item[](\ref{T''})$\implies$(\ref{D}): We prove the contrapositive. Suppose $S\in \cU$ and that $\Upsilon_\cU(S)$ is not down-closed. Then there is $p\in P$ with $p\leq \bv S$ and $p\notin
 \Upsilon_\cU(S)$. Since $p\leq \bv S$ we have $p\in \Gamma_\cU(S)$, and thus $\pd\cap \Gamma_\cU(S) = \pd$. However, by
 lemma \ref{L:fix}(1),  if $p\in\Gamma_\cU(\pd\cap S^\downarrow)$ then $p = \bv(\pd\cap
 S^\downarrow)$, and it would then follow from lemma \ref{L:fix}(2) that $\pd\cap S^\downarrow\in \cU^+$.
 But this would imply that $p\in \Upsilon_\cU(S)$, contradicting our contrary assumption. We conclude that $\neg$(\ref{D})$\implies\neg$(\ref{T''}).

\item[](\ref{D})$\implies$(\ref{U}): This follows directly from lemma \ref{L:down}.

\item[](\ref{U})$\implies$(\ref{i}): Suppose that (\ref{U}) holds but $\Upsilon_\cU(S)$ is not the smallest $\cU$-ideal containing $S$ for some $S\subseteq P$.  Let
 $\alpha$ be the smallest ordinal such that there is $p\in \Gamma^\alpha_\cU(S)\setminus \Upsilon_\cU(S)$. Then
 $\alpha$ must be a successor ordinal. So there is $T\subseteq \Gamma^{\alpha-1}_\cU(S)$, with $T\in \cU$ and $p\leq
 \bv T$. It follows from the minimality of $\alpha$ that $T\subseteq \Upsilon_\cU(S)$. Now, by the assumption of (\ref{U}) we have
 that $\Upsilon_\cU(T)$ is the smallest $\cU$-ideal of $P$ containing $T$, and thus $p\in \Upsilon_\cU(T)$. So there
 must be $T'\subseteq T^\downarrow$, with $T'\in\cU^+$ and $\bv T' = p$.

Since $T'\subseteq T^\downarrow$, and since $\Gamma^{\alpha-1}_\cU(S)$ is down-closed,  it follows again
 from the minimality of $\alpha$ that $T'\subseteq \Upsilon_\cU(S)$. Thus $p\in \Upsilon^2_\cU(S)$ (using the notation of definition \ref{D:U}), and by lemma \ref{L:equal} this is $\Upsilon_\cU(S)$. This contradicts the choice of $p$, so we conclude that (\ref{U})$\implies$(\ref{i}) as required.

\item[](\ref{i})$\implies$(\ref{f}): Let $I$ be an indexing set, let $C_i$ be $\cU$-ideals for all $i\in I$, and let $D$ be another $\cU$-ideal. We must show that $\Upsilon_\cU(\bigcup_I (D\cap C_i)) = D\cap \Upsilon_\cU(\bigcup_I C_i)$. Note that the left to right inclusion is automatic. So let $x\in D\cap \Upsilon_\cU(\bigcup_I C_i)$. Then there is
 $S\subseteq \bigcup_I C_i$ with $S\in \cU^+$ and $\bv S = x$. But then $S \subseteq D$, as $x\in D$. So $x \in \Upsilon_\cU(\bigcup_I D\cap C_i)$ as required.

\item[](\ref{i})$\iff$(\ref{des}): If $S\in\cU$, then $p\leq \bv S\implies p\in \Gamma_\cU(S)$, so if $\Gamma_\cU=\Upsilon_\cU$ we must have $p\in \Upsilon_\cU(S)$, and thus $p= \bv T$ for some $T\in\cU^+$ with $T\subseteq S^\downarrow$, by definition of $\Upsilon_\cU$.

Conversely, if $\Gamma_\cU\neq\Upsilon_\cU$ there is $X\subseteq P$ and a least $\alpha$ with $\Upsilon_\cU(X)\subset
 \Gamma_\cU^\alpha(X)$. Thus, as $\alpha$ must be a successor ordinal, there is $p\in P$, and $S\in\cU$, with
 $S\subseteq \Gamma_\cU^{\alpha-1}(X)$ and $p\leq \bv S$, but $p\neq \bv T$ for all $T\in\cU^+$ with
 $T\subseteq\Gamma_\cU^{\alpha-1}(X)$. So, in particular, $p\leq \bv S$, but $p\neq \bv T$ for all $T\in\cU^+$ with
 $T\subseteq S^\downarrow$.
\end{enumerate}[(1)]

Finally, (\ref{f}) trivially implies (\ref{d}), and, conversely, $\neg(\ref{T''})$ implies $\neg(\ref{d})$, by definition of the radius $\sigma$. That (\ref{f})$\iff$(\ref{h}) follows easily from proposition \ref{P:arrow}.

\end{proof}

Theorem \ref{T:gen} inspires the following definition.

\begin{defn}
If $\cU$ is a join-specification then we say it is \textup{\textbf{frame-generating}} if it satisfies the equivalent conditions of theorem \ref{T:gen}.
\end{defn}

We define the following notation for convenience.

\begin{defn}[$\cU_\alpha$, $\cU_\infty$]\label{D:Ua}
Given $P$, if $\alpha$ is a cardinal with $\alpha\geq 2$, define
\[\cU_\alpha=\{S\in\wp(P)\setminus\{\emptyset\} : |S|<\alpha\text{ and $\bv S$ exists in }P \}.\]
Similarly, define
\[\cU_\infty=\{S\in\wp(P)\setminus\{\emptyset\} :\text{$\bv S$ exists in P} \}.\]
\end{defn}

In \cite{HicMon84}, the focus is entirely on join-specifications of form $\cU_\kappa$, where $\kappa$ is a regular cardinal. In this situation we can strengthen definition \ref{D:des} as follows.

\begin{defn}\label{D:Sdes}
If $\cU$ is a join-specification for $P$, then $P$ has the \textup{\textbf{strong $\cU$-descent property}} if, for every down-closed $S\in \cU$ and $p\leq \bv S$, there is $T\in\cU$ with $T\subseteq S^\downarrow$ and $p=\bv T$.
\end{defn}

\cite[theorem 2.7]{HicMon84} is now included as a special case of our theorem \ref{T:gen} via the following lemma.

\begin{lemma}
Let $P$ be a poset, and let $\kappa$ be a regular cardinal. Then $P$ has the strong $\cU_\kappa$-descent property if and only if $\cU_\kappa$ is frame-generating.
\end{lemma}
\begin{proof}
If $P$ has the strong $\cU$-descent property, then it certainly has the $\cU$-descent property, so one direction follows immediately from the definition of frame-generating. For the converse, assume $\cU_\kappa$ is frame-generating. Let $S\in \cU_\kappa$ be down-closed, and
 suppose $p\leq \bv S$.  Since $\cI_{\cU_\kappa}$ is a frame, we have
\[\pd = \pd\cap \Gamma_{\cU_\kappa}(S) = \Gamma_{\cU_\kappa}(\pd\cap S).\]
So $p = \bv(\pd\cap S)$, and $\pd\cap S$ is clearly in $\cU_\kappa$.
\end{proof}

Some equivalent conditions for a complete lattice to be a frame are presented as \cite[theorem 2.1]{ErnWil83}, phrased as statements about closure systems and their associated closure operators. This theorem, modulo some fiddling, amounts to variations on proposition \ref{P:arrow}.

Note that \cite[theorem 2.1]{ErnWil83} is presented without proof, and directs the reader to \cite{Ern82} for details. The latter document is unfortunately not easily available, but proofs can be reconstructed with the kind of arguments used in the proof of proposition \ref{P:arrow}.

Example \ref{E:strength} below demonstrates that the restriction of corollary \ref{C:arrow} to finite $J$ is necessary. First we note the following fact, which is a mild generalization of \cite[theorem 4.4]{Egr17b}.

\begin{lemma}\label{L:MD}
Let $\alpha$ be a cardinal with $\alpha\geq 2$, and let $\cU$ be a join-specification for $P$ such that $\cU\subseteq \cU_\alpha\subseteq \cU^+$. Suppose $P$ has the property that for all $p\in P$ and $T\in \cU$, whenever $p\wedge \bv T$ is defined, $\bv_T(p\wedge t)$ is also defined, and the two are equal. Then $\cI_\cU$ is a frame.
\end{lemma}
\begin{proof}
Let $T\in\cU$ and $p\in P$ with $p\leq \bv T$. Then, appealing to proposition \ref{P:pres} and the definitions of join and meet in $\cI_\cU$, we have
\[\pd\cap \Gamma_\cU(T) = \eta(p)\wedge \bv \eta[T] = \eta(p\wedge \bv T) = \eta(\bv_T(p\wedge t)),\]
and
\[\Gamma_\cU(\pd\cap T^\downarrow)=\bv_T(\eta(p\wedge t)).\]
Now, we must have $|T|<\alpha$, as $\cU\subseteq\cU_\alpha$, and $|\{p\wedge t:t\in T\}|=|T|$, so $\{p\wedge t:t\in T\}\in \cU^+$. Thus $\eta(\bv_T(p\wedge t))= \bv_T(\eta(p\wedge t))$, by proposition \ref{P:pres} and the definition of $\cU^+$, and so $\cI_\cU$ is a frame by theorem \ref{T:gen}(\ref{T''}).
\end{proof}

\begin{ex}\label{E:strength}
Let $P$ be the poset in figure \ref{F:pos}. Then $P$ is a complete lattice, and is also clearly distributive. Let $p\in P$ and let $Y\subseteq P$. We aim first to show that $p\wedge \bv Y= \bv_Y(p\wedge y)$, and consequently that $P$ is a frame.

If $\bv Y=\bot$ then this is trivial,
 and the remaining possibilities divide into two cases, as either $\bv Y = s_n$ for some $n\in \omega$, or $\bv Y= x_m$
 for some $m\in \omega$.

In the first case, it follows that $Y\subseteq \{\bot,s_0,s_1,\ldots\}$, and that $\bv Y\in Y$. In
 this case $p\wedge \bv Y= p\wedge s_n =  \bv_Y(p\wedge y)$ as required.

Alternatively, suppose $\bv Y= x_m$ for some $m\in \omega$. Then either $x_m\in  Y$, in which case
\[p\wedge \bv Y = p\wedge x_m \leq \bv_Y(p\wedge y)\leq p\wedge \bv Y,\]
or $x_m\notin Y$, in which case $s_m\in Y$. There must be a smallest $k$ such that $x_k\in Y$, and then $s_m \vee x_k = x_m$. So
\[p\wedge \bv Y = p\wedge (s_m\vee x_k)= (p\wedge s_m)\vee (p\wedge x_k)\leq \bv_Y(p\wedge y)\leq p\wedge \bv Y. \]
Thus $P$ is a frame. So, by lemma \ref{L:MD}, $\cI_{\cU_\infty}$ is also a frame. However, if we define $S_n= s_0^\downarrow \cup x_n^\downarrow$ for all $n\geq 1$, then
 $\Gamma_{\cU_\infty}(S_n) = P$ for all $n$, but $\bigcap_\omega S_n = s_0^\downarrow$, and $\Gamma_{\cU_\infty}(s_0^\downarrow) =
 s_0^\downarrow$. Thus $\Gamma_{\cU_\infty}(\bigcap_\omega S_n)\subset \bigcap _\omega \Gamma_{\cU_\infty}(S_n)$.
\end{ex}

\begin{figure}[htbp]
\centering
\scalebox{0.7}{\xymatrix{ & \bullet{x_0}\ar@{-}[dr]\ar@{-}[dl] \\
\bullet{s_0}\ar@{-}[dr] & & \bullet{x_1}\ar@{-}[dr]\ar@{-}[dl] \\
 & \bullet{s_1}\ar@{-}[dr] & & \bullet{x_2}\ar@{-}[dr]\ar@{-}[dl] \\
& & \bullet{s_2}\ar@{-}[dr] & & \bullet{x_3}\ar@{..}[dr]\ar@{-}[dl] \\
& & & \bullet{s_3}\ar@{..}[dr] & & & & & \\
& & & & & & & & & &\\
& & & & &  \bullet{\bot}
}}
\caption{}
\label{F:pos}
\end{figure}

\section{A `local' categorical perspective: $\JF$, $\JF^+$ and $\Frm_P$}\label{S:cat}

As mentioned previously, given a poset $P$, the set of join-specifications for $P$ forms a complete lattice whose joins and meets are unions and intersections respectively. When we focus on frame-generating join-specifications, the situation is more interesting, though it turns out that the frame-generating join-specifications still form a complete lattice, as do the maximal frame-generating join-specifications. The goal in this and the next section is to define the joins and meets in these lattices, and explore their relationships, as categories, with a suitably defined category of frame-completions of $P$.

\begin{lemma}\label{L:init}
Let $\cU_1$, $\cU_2$ and $\cU_3$ be frame-generating join-specifications for $P$ with $\cU_1\subseteq\cU_2\subseteq\cU_3$. Then there are frame morphisms $\phi_1$, $\phi_2$ and $\phi$ such that the diagram in figure
 \ref{F:factor} commutes. Moreover, $\phi_1$, $\phi_2$ and $\phi$ are the unique frame morphisms such that
 corresponding triangles with $\eta$ maps commute.
\end{lemma}
\begin{proof}
This follows directly from proposition \ref{P:arrow}.
\end{proof}

\begin{figure}
\begin{minipage}[b]{0.5\textwidth}
\centering\scalebox{1}{
\xymatrix{
P\ar[d]_\eta\ar[dr]_\eta\ar[drr]^\eta \\
 \cI_{\cU_1}\ar[r]_{\phi_1}\ar@/_1.5pc/[rr]_{\phi} & \cI_{\cU_2}\ar[r]_{\phi_2} & \cI_{\cU_3}
}}
\caption{}
\label{F:factor}
\end{minipage}
\hfill
\begin{minipage}[b]{0.5\textwidth}
\centering\scalebox{1}{
\xymatrix{ P\ar[r]^{e_1}\ar[dr]_{e_2} & L_1\ar[d]^f \\
& L_2
}}
\caption{}
\label{F:arrows}
\end{minipage}
\end{figure}

Given a poset $P$, recall that a join-completion of $P$ is an order embedding $e:P\to L$ where $L$ is a complete lattice and $e[P]$ is join-dense in $L$. Fixing $P$, define $\Frm_P$ to be the category whose objects are join-completions $e:P\to L$ such that $L$ is a frame, and whose arrows are commuting diagrams as in figure \ref{F:arrows}, where $f$ is a frame morphism. Note that such a morphism $f$ is necessarily onto.

So $\Frm_P$ is a subcategory of the coslice category $P\downarrow \Pos$, where $\Pos$ is the category of posets and monotone maps. $\Frm_P$ is clearly neither full nor essentially wide as a subcategory of $P\downarrow \Pos$.
 It is the case, however, that $\Frm_P$ is replete in $P\downarrow \Pos$. To see this note that if $e:P\to L$ is in
 $\Frm_P$ and $f:P\to Q$ is in $P\downarrow \Pos$, then an isomorphism between $e$ and $f$ in $P\downarrow \Pos$ forces
 $Q$ to be a frame and $f$ to be a join-completion. If $\cA(P) = \cI_{B_P}$, it follows from proposition \ref{P:arrow} that $\eta:P\to \cA(P)$ is an initial object in $\Frm_P$. Note that $\Frm_P$ is a thin category.

Let $\JF$ be the category whose objects are frame-generating join-specifications for $P$, and whose maps are just
 inclusions. Note that $\JF$ is never empty, as $B_P$ is always frame-generating. We can define a map from $\JF$ to $\Frm_P$ by taking $\cU$ to $\eta:P\to\cI_\cU$, and taking
 inclusions to the induced maps $\phi$ from lemma \ref{L:init}. Using lemma \ref{L:init},
 it's straightforward to prove that this map defines a functor, which we call $F_P$. Note that $F_P$ is full and faithful,
 as fullness follows from lemma \ref{L:init}, and faithfulness follows trivially from the fact that $\JF$ is thin. Thus
 $\JF$ is equivalent as a category to its image under $F_P$.

Conversely, given a join-completion $e:P\to L$ we obtain a closure operator, $\Gamma_e$ say, on $P$ by defining the
 closed sets to be all sets of form $e^{-1}(\xd)$ for $x\in L$ (as in the definition of $\gamma$ in proposition \ref{P:closure}). We can then define the join-specification
 $\cU_{\Gamma_e}$ using definition \ref{D:UG}, and subsequently we can construct $\eta:P\to \cI_{\cU_{\Gamma_e}}$. We
 can think of $L$ as being a complete lattice of $\cU_{\Gamma_e}$-ideals, while $\cI_{\cU_{\Gamma_e}}$ is the complete
 lattice of \emph{all}
 $\cU_{\Gamma_e}$-ideals, and thus likely to be larger. Moreover, as example \ref{E:notMod} demonstrates below, even if
 $L$ is (completely) distributive, the lattice $\cI_{\cU_{\Gamma_e}}$ may not even be modular. Thus the most obvious
 candidate for a functor from $\Frm_P$ to $\JF$ fails to even be well defined. Nevertheless, such a functor does exist, as
 we shall see.

\begin{ex}\label{E:notMod}
Let $\phi:P\to L$ be the join-completion as defined in figure \ref{F:notModL}. Here the elements of $\phi[P]$ are represented with $\bullet$, and the elements of $L\setminus \phi[P]$ are represented with $\circ$. Certain elements are labelled for reference later. Inspection reveals that $L$ is distributive. Let $\Gamma_\phi$ be the standard closure operator associated with $\phi$, and for brevity write $\cU_\phi$ for $\cU_{\Gamma_\phi}$.

We note that $\phi$ preserves all joins that are defined in $P$, so $\cI_{\cU_\phi}$ is just the set of all downsets of $P$ that are closed under existing joins. Consider the sets $b^\downarrow$, $d^\downarrow$, $e^\downarrow$,
 $\{b,c\}^\downarrow$, $\{a,b,c\}^\downarrow$. These are all closed under joins existing in $P$, so they are all
 elements of $\cI_{\cU_\phi}$. Moreover, these elements as a sublattice of $\cI_{\cU_\phi}$ form the pentagon, as shown
 in figure \ref{F:notModI}. This can be seen by noting that the $\Gamma_\phi$-closures of the sets $\{a,b,c,d\}$ and
 $\{b,c, d\}$ must both be $e^\downarrow$, as the join of $b$ and $d$ is $e$ in $P$. Thus $\cI_{\cU_\phi}$ is not modular.

As a bonus, since $\cU_\phi$ is the largest join-specification for $P$, it follows that there is no $\cU\in\JF$ with $\cU_\phi\subseteq \cU$. This is relevant as it means we can't hope to `approximate $\cU_\phi$ from above' with members of $\JF$.
\end{ex}

\begin{figure}
\begin{minipage}[b]{0.5\textwidth}
\centering\scalebox{1}{
\xymatrix{
\circ\ar@{-}[d]\ar@{-}[dr]\\
 \bullet\ar@{-}[dr] & \bullet_e\ar@{-}[d]\ar@{-}[dr] \\
 & \circ\ar@{-}[dl]\ar@{-}[d]\ar@{-}[dr] & \bullet_d\ar@{-}[d] \\
 \bullet_a\ar@{-}[d]\ar@{-}[dr] & \bullet_b\ar@{-}[dl]\ar@{-}[dr] & \bullet_c\ar@{-}[dl]\ar@{-}[d] \\
 \bullet\ar@{-}[dr] & \bullet\ar@{-}[d] & \bullet\ar@{-}[dl] \\
 & \circ
}}
\caption{}
\label{F:notModL}
\end{minipage}
\hfill
\begin{minipage}[b]{0.5\textwidth}
\centering\scalebox{1}{
\xymatrix{
& e^\downarrow\ar@{-}[dl]\ar@{-}[ddr] \\
\{a,b,c\}^\downarrow\ar@{-}[dd] \\
& & d^\downarrow\ar@{-}[ddl] \\
\{b,c\}^\downarrow\ar@{-}[dr] \\
& c^\downarrow
}}
\caption{}
\label{F:notModI}
\end{minipage}
\end{figure}

Actually we can deduce that the lattice $\cI_{\cU_\phi}$ from example \ref{E:notMod} must fail to be distributive from general considerations. Corollary \ref{C:unique} below says that there is, up to isomorphism, at most one distributive $\cU$-preserving join-completion of $P$ whenever $P$ is finite. Since $L$ has this property, and since $\cI_{\cU_\phi}\not\cong L$, it follows that $\cI_{\cU_\phi}$ cannot be distributive.

\begin{lemma}\label{L:unique}
Let $D$ be a finite distributive lattice. Let $J(D)$ be the set of join-irreducible elements of $D$. Let $\cA (J(D))$ be the lattice of downsets of $J(D)$. Then $D\cong \cA(J(D))$.
\end{lemma}
\begin{proof}
This is well known. See e.g. \cite[theorem 107]{Gra11}.
\end{proof}

\begin{cor}\label{C:unique}
Let $P$ be a finite poset without a bottom element, and let $\cU$ be a join-specification for $P$. Then there is, up to isomorphism, at most one join-completion $e:P\to D$ such that $e$ is $\cU$-preserving and $D$ is distributive.
\end{cor}
\begin{proof}
By lemma \ref{L:unique}, if $D$ exists then $D \cong \cA(J(D))$. As $D$ is join-generated by $e[p]$ we must have $J(D)\subseteq e[P]$. In fact,
\[J(D)=\{e(p): p\neq \bv S \text{ for all }S\in \cU^+\setminus B_P^+\},\]
so $J(D)$ is determined by $\cU$, and the result follows.
\end{proof}

If $P$ has a bottom element then corollary \ref{C:unique} does not hold, but it almost does, as in this case there are at most two such distributive join-completions. This difference comes down to whether or not the bottom of $P$ is join-irreducible in $D$. If not then the bottom of $D$ is the original bottom of $P$, and if so then $D$ will have a new bottom corresponding to the empty join.

We now investigate some properties of the category $\JF$ and the functor $F_P$. First we define an interesting subcategory of $\JF$.

\begin{defn}\label{D:JFP}
Define $\JF^+$ to be the full subcategory of $\JF$ containing all maximal frame-generating join-specifications (recall definition \ref{D:max}).
\end{defn}

Note that $B_P^+$ is always frame-generating, so $\JF^+$ is never empty.
The following theorem provides some information about the structures of $\JF$ and $\JF^+$. It turns out that $\JF$ behaves well with respect to unions, and $\JF^+$ behaves well with respect to intersections, but, in general, neither behaves well with respect to both.

\begin{thm}\label{T:lims}
Let $I$ be an indexing set and let $\cU_i\in\JF$ for all $i\in I$. Then:
\begin{enumerate}[(1)]
\item $\bigcup_I \cU_i\in \JF$.\label{limsU}
\item Even finite unions of maximal frame-generating join-specifications may not be maximal.\label{limsM}
\item If $\cU_i\in \JF^+$ for all $i\in I$ then $\bigcap_I \cU_i\in \JF^+$.\label{limsI}
\item Even finite intersections of frame-generating join-specifications may not be frame-generating.\label{limsF}
\end{enumerate}
\end{thm}
\begin{proof}\mbox{}
\begin{enumerate}[(1)]
\item Let $S\in \bigcup_I \cU_i$, let $p\in P$, and suppose $p\leq \bv S$. Since $S\in \cU_j$ for some $j\in I$, and since $\cU_j\in
 \JF$, we have $p\in \Upsilon_{\cU_j}(S)$, by theorem \ref{T:gen}(\ref{D}). So,
 since $\Upsilon_{\cU_j}(S)\subseteq \Upsilon_{\bigcup_I\cU_i}(S)$, we also have $p\in \Upsilon_{\bigcup_I\cU_i}(S)$. Thus
 $\Upsilon_{\bigcup_I\cU_i}(S)$ is down-closed, and so $\bigcup_I \cU_i\in\JF$, again by theorem \ref{T:gen}(\ref{D}).
\item This is demonstrated in example \ref{E:noUnion} below.
\item By lemma \ref{L:capmax}, we have $(\bigcap_I \cU_i)^+=\bigcap_I \cU_i$. The strategy now is to prove that
 $\Upsilon_{\bigcap_I \cU_i}(S)$ is down-closed whenever $S\in \bigcap_I \cU_i$, and then to appeal to theorem
 \ref{T:gen}(\ref{D}).

So let $S\in \bigcap_I \cU_i$, and let $p\in P$ with $p<\bv S$. As $\Upsilon_{\cU_i}(S)$ must be down-closed for all
 $i\in I$ by theorem \ref{T:gen}(\ref{D}), for each $i\in I$ there is $T_i\in \cU_i=\cU_i^+$ with $T_i\subseteq S^\downarrow$
 and $\bv T_i = p$.

Choose some arbitrary $k\in I$. Again it follows from
 theorem \ref{T:gen}(\ref{D}) that $\Upsilon_{\cU_k}(T_k)$ is down-closed, and consequently that $T_i\subseteq
 \Upsilon_{\cU_k}(T_k)\cap S^\downarrow$ for all $i\in I$. So
 \[\Upsilon_{\cU_i}(\Upsilon_{\cU_k}(T_k)\cap S^\downarrow) = \pd,\]
and thus $\Upsilon_{\cU_k}(T_k)\cap
 S^\downarrow\in\cU_i^+=\cU_i$, by lemma \ref{L:fix}. So $\Upsilon_{\cU_k}(T_k)\cap S^\downarrow\in\bigcap_I
 \cU_i$, and so $p\in \Upsilon_{\bigcap_I\cU_i}(S)$. Thus
 $\Upsilon_{\bigcap_I\cU_i}(S)$ is down-closed as required.

\item This is demonstrated in example \ref{E:strict} below.
\end{enumerate}
\end{proof}

\begin{ex}\label{E:noUnion}
Let $P$ be the poset in figure \ref{F:noUnion}. Let $B_P$ be as in definition \ref{D:B}, let $\cU_1= B_P\cup
 \{\{a,b\}\}$, and let $\cU_2 = B_P\cup\{\{b,c\},\{d,e\}\}$. Using theorem \ref{T:gen}(\ref{D}) it's easy to check
 that $\cU_1$ and $\cU_2$ are frame-generating. Inspection reveals, after a little thought, that $\{a,b,c\}\notin
 \cU_1^+\cup \cU_2^+$. Thus $\cU_1^+\cup\cU_2^+$ is not maximal, as $\{a,b,c\}\in (\cU_1^+\cup \cU_2^+)^+$.
\end{ex}

\begin{figure}[htbp]
\centering
\scalebox{1}{\xymatrix{
& & & \bullet_f\ar@{-}[dl]\ar@{-}[d] \\
& & \bullet_d\ar@{-}[dl]\ar@{-}[d] & \bullet_e\ar@{-}[dl]\ar@{-}[d]\\
& \bullet_a & \bullet_b & \bullet_c
}}
\caption{}
\label{F:noUnion}
\end{figure}

\begin{ex}\label{E:strict}
Let $P$ be the poset in figure \ref{F:strict}. Let
\[\cU_1 =B_P\cup \{\{a,b\},\{b, c\},\{c,d\},\{b,g\},\{c,e\},\{a,b,c,d,e,g\}\}\]
and let
\[\cU_2 = B_P \cup\{\{a,b\},\{b,c\},\{c,d\},\{a,b,c\},\{b,c,d\},\{a,b,c,d,e,g\}\}.\]

Using theorem \ref{T:gen}(\ref{D}), it's straightforward to check that both $\cU_1$ and $\cU_2$ are frame-generating.

However, $\cU_1\cap \cU_2 = B_P\cup \{\{a,b\},\{b,c\},\{c,d\},\{a,b,c,d,e,g\}\}$, and so $h,i\notin \Upsilon_{\cU_1\cap \cU_2}(\{a,b,c,d,e,g\})$. So $\Upsilon_{\cU_1\cap \cU_2}(\{a,b,c,d,e,g\})$ is not down-closed, and thus
$\cU_1\cap \cU_2$ is not frame-generating, using theorem \ref{T:gen}(\ref{D}) again.

This example also demonstrates that the inclusion $(\bigcap_I \cU_i)^+\subseteq \cU_{\bw_I \Gamma_{\cU_i}}$ in the
 proof of lemma \ref{L:capmax} can be strict, as when we have equality, $(\bigcap_I \cU_i)^+$ is equal to the
 intersection $\bigcap_I\cU_i^+$ of maximal join-specifications. If these are all frame-generating, then $(\bigcap_I
 \cU_i)^+$ will also be frame-generating, by theorem \ref{T:lims}(\ref{limsI}). It would then follow that $\bigcap_I \cU_i$
 is also frame-generating, as it generates the same lattice. Thus, in example \ref{E:strict} there must be $S\in
 (\cU_1^+\cap\cU_2^+)\setminus (\cU_1\cap\cU_2)^+$, as otherwise $\cU_1\cap\cU_2$ would be frame-generating. We can just take, e.g. $S = \{a,b,c\}$ for a concrete example.
\end{ex}

\begin{figure}[htbp]
\centering
\scalebox{1}{\xymatrix{
& & & \bullet_j\ar@{-}[dl]\ar@{-}[d] \\
& & \bullet_h\ar@{-}[dl]\ar@{-}[d] & \bullet_i\ar@{-}[dl]\ar@{-}[d]\\
& \bullet_e\ar@{-}[d]\ar@{-}[dl] & \bullet_f\ar@{-}[d]\ar@{-}[dl] & \bullet_g\ar@{-}[d]\ar@{-}[dl]\\
\bullet_a & \bullet_b & \bullet_c &\bullet_d}}
\caption{}
\label{F:strict}
\end{figure}

\section{Completeness properties of $\JF$ and $\JF^+$}\label{S:comp}

Theorem \ref{T:comps} below proves that both $\JF$ and $\JF^+$ are complete as lattices, and describes their meet and join structures. First though we show that every join-specification can be reduced to a frame-generating one in a canonical way. We need the following lemma.

\begin{lemma}\label{L:remove}
Let $\cU$ be a join-specification, let $S\in\cU$, and let $X\subseteq \cU\setminus\{S\}$. Suppose $\Upsilon_\cU(S)$ is not down-closed. Then $\Upsilon_{\cU\setminus X}(S)$ is not down-closed.
\end{lemma}
\begin{proof}
$\Upsilon_\cU(S)$ fails to be down-closed if and only if there is $p<\bv S$ such that $p\neq \bv S'$ for all $S'\subseteq S^\downarrow$ with $S'\in\cU^+$. If there is no $S'\in \cU^+$ with the desired property, then there is certainly no such $S'\in(\cU\setminus X)^+$, so the result follows.
\end{proof}

\begin{prop}\label{P:core}
If $\cU$ is a join-specification, then there is a largest $\cU'\subseteq \cU$ with $\cU'\in\JF$.
\end{prop}
\begin{proof}
We will construct $\cU'$ using recursion. We make the following definitions:
\begin{itemize}
\item $U_0 = \cU$.
\item Let $\alpha+1$ be a successor ordinal. If $U_\alpha$ is frame-generating, then define $U_{\alpha+1} =U_\alpha$. Otherwise, let $X= \{S\in U_\alpha: \Upsilon_{U_\alpha}(S)$ is not down-closed$\}$, and define $U_{\alpha+1} = U_\alpha\setminus X$.
\item Let $\lambda$ be a limit ordinal. Define $U_\lambda = \bigcap_{\beta<\lambda} U_\beta$.
\end{itemize}
Let $\kappa=|\wp(P)|^+$, the successor cardinal of $|\wp(P)|$. Then $U_\kappa = U_{\kappa'}$ for all $\kappa'\geq \kappa$. So we can define $\cU'=U_\kappa$.

Now, $\cU'$ is certainly frame-generating, as the process of its construction leaves it with no sets that violate theorem \ref{T:gen}(\ref{D}). That it is the largest such subset of $\cU$ follows from lemma \ref{L:remove}, which implies that `problematic' sets violating theorem \ref{T:gen}(\ref{D}) do not become unproblematic by the removal of other sets, and thus it is necessary to remove every set that becomes problematic at any point in the construction. The process must terminate, and the limiting join-specification is $B_P$.
\end{proof}

\begin{defn}[$\cU^-$]\label{D:core}
Using proposition \ref{P:core}, given a join-specification $\cU$ define $\cU^-$ to be the largest frame-generating join-specification contained in $\cU$.
\end{defn}

\begin{lemma}\label{L:alwaysMax}
Let $e:P\to L$ be a join-completion, let $\Gamma_e$ be the standard closure operator associated with $e$, and let $\cU_e = \cU_{\Gamma_e}$. Then $\cU_e^-$ is maximal.
\end{lemma}
\begin{proof}
    Note that $\cU_e$ is always maximal, by proposition \ref{P:adj}(3), so
    $\cU_e^-\subseteq \cU_e\implies (\cU_e^-)^+\subseteq \cU_e^+ = \cU_e$.
    Thus $(\cU_e^-)^+$ is a frame-generating join-specification contained in $\cU_e$, and so must be equal to $\cU_e^-$, which is the largest such join-specification.
\end{proof}

\begin{thm}\label{T:comps}\mbox{}
\begin{enumerate}[(1)]
\item $\JF$ is complete and cocomplete, and is thus a complete lattice. Non-empty joins are given by unions, and the top and bottom are $\bigcup_{\cU\in\JF}\cU$ and $B_P$ respectively.\label{compsJF}
\item $\JF^+$ is complete and cocomplete, and is thus a complete lattice. Non-empty meets are given by intersections, and the top and bottom are $\bigcup_{\cU\in\JF}\cU$ and $B_P^+$ respectively. \label{compsJFP}
\end{enumerate}
\end{thm}
\begin{proof}\mbox{}
\begin{enumerate}[(1)]
\item $B_P$ is clearly the bottom of $\JF$, as it produces the lattice of downsets, $\cA(P)$, which is always completely distributive. The terminal object of $\JF$ is the union $\bigcup_{\cU\in \JF} \cU$, which is in $\JF$ by theorem \ref{T:lims}(\ref{limsU}). That non-empty joins correspond to unions follows from the same theorem. This shows that $\JF$ is cocomplete.

Now, let $I$ be an indexing set and let $\cU_i\in\JF$ for each $i\in I$. The limit for (the diagram associated with) $\{\cU_i:i\in I\}$ is given by $(\bigcap_I\cU_i)^-$ as in definition \ref{D:core}, and thus $\JF$ is complete.

\item We can deduce that $\JF^+$ is a complete lattice, and thus complete and cocomplete as a category, by an application of the Knaster-Tarski theorem \cite{Tar55a} to $\JF$, which we know from part \ref{compsJF} to be a complete lattice, and the functor $\G\circ\F$, which necessarily preserves order.

This result also follows directly from theorem \ref{T:lims}(\ref{limsI}), and the fact that if $\cU_i\in\JF$ for all $i\in I$ we must have that $(\bigcup_I \cU_i)^+$ is the smallest maximal frame-generating join-specification containing
 $\bigcup_I \cU_i$. From the latter observation we also see that the top of $\JF^+$ is $(\bigcup_{\cU\in\JF^+}\cU)^+$.
 Now, $\cU\subseteq \cU^+$ for all $\cU\in\JF$, so $\bigcup_{\cU\in\JF}\cU\subseteq (\bigcup_{\cU\in\JF^+}\cU)^+$, and
 the other inclusion is automatic as $(\bigcup_{\cU\in\JF^+}\cU)^+\in\JF$.
That the bottom of $\JF^+$ is $B_P^+$ is automatic from the definitions.
\end{enumerate}
\end{proof}

\begin{cor}\label{C:ref}
Let $\iota$ be the inclusion of $\JF^+$ into $\JF$, and let $M$ be the map from $\JF$ to $\JF^+$ defined by $M(\cU)=\cU^+$. Then $M\dashv \iota$.
\end{cor}
\begin{proof}
Since $\JF$ and $\JF^+$ are complete lattices, this follows from the fact that for $\cU_1\in \JF$ and $\cU_2\in \JF^+$ we have
\[\cU_1^+\subseteq \cU_2 \iff \cU_1\subseteq \cU_2.\]
\end{proof}

\begin{cor}\label{C:inc}
The inclusion functor of $\JF^+$ into $\JF$ preserves limits, but does not usually preserve colimits, including the initial object.
\end{cor}
\begin{proof}
Preservation of limits follows immediately from corollary \ref{C:ref} and properties of right adjoints. That the inclusion does not usually preserve initial objects follows from theorem \ref{T:comps} and the fact that $B_P$ is usually a strict subset of $B_P^+$. That the inclusion may not preserve non-empty colimits is essentially a reformulation of theorem \ref{T:lims}(\ref{limsM}).
\end{proof}

\begin{cor}
$\JF^+$ is a reflective subcategory of $\JF$.
\end{cor}
\begin{proof}
This is just a restatement of corollary \ref{C:ref}.
\end{proof}

Theorem \ref{T:pres} below examines the preservation of limits and colimits by the functor $F_P:\JF\to \Frm_P$ and its restriction to $\JF^+$. First we will construct the right adjoint for the functor $F_P$. We will need some technical results. Recall definition \ref{D:JC} for the definition of $\JC$.

\begin{lemma}\label{L:contained}
Let $e_1:P\to L_1$ and $e_2:P\to L_2$ be in $\JC$. Let $f:L_1\to L_2$ be a completely join-preserving map fixing $P$. For $i\in\{1,2\}$ let $\Gamma_{e_i}$ be the standard closure operator associated with $e_i$. Then $\cU_{\Gamma_{e_1}}\subseteq \cU_{\Gamma_{e_2}}$.
\end{lemma}
\begin{proof}
Let $S\subseteq P$. Then $S\in \cU_{\Gamma_{e_1}}$ if and only if $\bv S$ exists and $e_1(\bv S) = \bv e_1[S]$. Then, given $S\in \cU_{\Gamma_{e_1}}$,  since $f$ is completely join-preserving and $f\circ e_1 = e_2$, we have
\[e_2(\bv S)= f\circ e_1(\bv S) = \bv f\circ e_1[S] = \bv e_2[S].\]
So $S\in \cU_{\Gamma_{e_2}}$.
\end{proof}

\begin{cor}\label{C:contained}
 Let $\cU\in\JF$, let $\eta$ be the canonical map from $P$ to $\cI_\cU$, let $e:P\to L$ be an object of $\JC$, and let $\Gamma_e$ be the standard closure operator associated with $e:P\to L$. Then:
\begin{enumerate}[(1)]
\item Let $f:\cI_\cU\to L$ be a completely join-preserving map fixing $P$.  Then $\cU^+\subseteq\cU_{\Gamma_e}$.
\item Let $g:L\to \cI_\cU$ be a completely join-preserving map fixing $P$.  Then $\cU_{\Gamma_e}\subseteq \cU^+$.
\end{enumerate}
\end{cor}
\begin{proof}
This follows easily from lemma \ref{L:contained}.
\end{proof}

\begin{defn}[$G_P$]
Given $e\in \JC$, define $G_P(e) = \cU_e^-$, where $\cU_e$ is short for $\cU_{\Gamma_e}$, and $\cU_{\Gamma_e}^-$ is defined as in definition \ref{D:core}.
\end{defn}

\begin{lemma}
$G_P$ defines a functor from $\JC$ to $\JF^+$.
\end{lemma}
\begin{proof}
    By lemma \ref{L:contained}, if $f:e_1\to e_2$ in $\JC$, then $\cU_{{e_1}}\subseteq \cU_{{e_2}}$, and
     so $\cU_{{e_1}}^-\subseteq \cU_{{e_2}}^-$. Moreover, as $\cU_e^-$ is maximal by lemma \ref{L:alwaysMax}, the result follows easily.
\end{proof}

\begin{thm}\label{T:term}
$F_P\dashv G_P$.
\end{thm}
\begin{proof}
Let $e:P\to L$ be an object of $\JC$. Let $(F_P\downarrow  e)$ be the comma category, where $e$ here is the constant functor.

Let $\cU_e = \cU_{\Gamma_e}$, let $\cU_e^-$ be as in definition \ref{D:core}, and let $\bar \eta$ be the canonical map from $P$ into $\cI_{\cU_e^-}$. Consider the map that takes $S\in\cI_{\cU_e^-}$ to $\Gamma_e(S)$. Since $\cU_e^-\subseteq \cU_e$ by definition, it follows from proposition \ref{P:arrow} that
this defines a completely join-preserving map $\phi:\cI_{\cU_e^-}\to L$ with $\phi\circ\bar{\eta} = e$.

We will show that $(F_P\downarrow e)$ has as terminal object the pair $(\cU_e^-, \phi)$. The result will then follows (see e.g. \cite[corollary 2.3.7]{Lein14} and take the dual).

 Objects of $(F_P\downarrow e)$ are pairs $(\cU,h)$ such that $\cU$ is a frame-generating join-specification, $h$ is a completely join-preserving map,
 and the diagram in figure \ref{F:comma1} commutes. An arrow from $(\cU_1,h_1)$ to $(\cU_2,h_2)$ in $(F_P\downarrow e)$
 is a commuting diagram as in figure \ref{F:comma3}, where the map from $\cI_{\cU_1}$ to $\cI_{\cU_2}$ is induced by the
 inclusion of $\cU_1$ into $\cU_2$ and lemma \ref{L:init}. As
 $\JC$ is a thin category, commutativity is guaranteed, so such a diagram exists if and only if $\cU_1\subseteq \cU_2$.

To see that $(\cU_e^-, \phi)$ is the terminal object as claimed, let $(\cU, h)$ be an object of $(F_P\downarrow e)$. Then, by corollary \ref{C:contained}, we must have $\cU^+\subseteq \cU_e$, and thus, as $\cU$ is frame-generating, we consequently have $\cU\subseteq \cU_e^-$. By lemma \ref{L:init}, this inclusion induces a frame morphism from $\cI_\cU$ to $\cI_{\cU_e^-}$ such that the diagram in figure \ref{F:comma2} commutes.
Thus there is a map from $(\cU, h)$ to $(\cU_e^-, \phi)$ in $(F_P\downarrow e)$. As $\JF$ is thin, this map must be unique, and so $(\cU_e^-, \phi)$ is terminal, as claimed.

\end{proof}

\begin{figure}
\begin{minipage}[b]{0.32\textwidth}
\centering\scalebox{1}{
\xymatrix{
\cI_\cU\ar[r]^h & L \\
P\ar[ur]_e\ar[u]^\eta
}}
\caption{}
\label{F:comma1}
\end{minipage}
\hfill
\begin{minipage}[b]{0.32\textwidth}
\centering\scalebox{1}{
\xymatrix{
\cI_{\cU_1}\ar[r]^{h_1}\ar@/^1.5pc/[rr] & L & \cI_{\cU_2}\ar[l]_{h_2} \\
& P\ar[u]_e\ar[ur]_{\eta_2}\ar[ul]^{\eta_1}
}}
\caption{}
\label{F:comma3}
\end{minipage}
\hfill
\begin{minipage}[b]{0.32\textwidth}
\centering\scalebox{1}{
\xymatrix{
\cI_\cU\ar[r]^h\ar@/^1.5pc/[rr] & L & \cI_{\cU_e^-}\ar[l]_{\phi} \\
& P\ar[u]_e\ar[ur]_{\bar\eta}\ar[ul]^{\eta}
}}
\caption{}
\label{F:comma2}
\end{minipage}
\end{figure}

\begin{cor}\mbox{}\label{C:adj}
Let $\iota$ be the inclusion of $\JF^+$ into $\JF$, and let $G'_P$ be the restriction of $G_P$ to $\Frm_P$. Then $F_P\circ \iota\dashv G_P$, $F_P\dashv G'_P$ and  $F_P\circ \iota\dashv G'_P$.
\end{cor}
\begin{proof}
That $F_P\circ \iota\dashv G_P$ follows from the proof of theorem \ref{T:term} by noting that $\cU_e^-$ is always maximal, by lemma \ref{L:alwaysMax}. The proofs for the adjunctions involving $G'_P$ are essentially identical to the proofs for the adjunctions for $G_P$. The only difference being that the phrase `completely join-preserving map' must be replaced by `frame morphism' throughout.
\end{proof}

\begin{thm}\label{T:pres}
Let $\iota$ be the inclusion of $\JF^+$ into $\JF$. Then:
\begin{enumerate}[(1)]
\item $F_P$ and $F_P\circ\iota$ preserve colimits. \label{pres3}
\item $F_P\circ \iota$ preserves non-empty limits, but may not preserve terminal objects.\label{pres1}
\item $F_P$ may not preserve even finite non-empty limits.\label{pres2}
\end{enumerate}
\end{thm}
\begin{proof}\mbox{}
\begin{enumerate}[(1)]
\item This follows immediately from corollary \ref{C:adj} and the fact that left adjoints preserve small colimits.
\item Let $I$ be an indexing set, and let $\cU_i\in \JF^+$ for all $i\in I$. Then, by theorem \ref{T:lims}(\ref{limsI}), the limit induced by $\{\cU_i:i\in I\}$ is $\bigcap_I\cU_i$, and this is preserved by
$\iota$, by corollary \ref{C:inc}. By definition, $F_P(\bigcap_I\cU_i)=\cI_{\bigcap_I\cU_i}$, and, similarly,
$F_P(\cU_i)=\cI_{\cU_i}$ for all $i\in I$. For each $i\in I$, let $\phi_i$ be the map from $\cI_{\bigcap_I \cU_i}$ to
$\cI_{\cU_i}$ induced by inclusion as in lemma \ref{L:init}. By this lemma, $\cI_{\bigcap_I \cU_i}$ and the $\phi_i$ maps induce a cone for the diagram associated with
 $\{\cI_{\cU_i}:i\in I\}$ in $\Frm_P$. Let $e:P\to L$ be an object of $\Frm_P$, and suppose for each $i\in I$ there is
 $f_i:L\to\cI_{\cU_i}$ fixing $P$, such that $L$ and the $f_i$ maps induce another such cone in $\Frm_P$. Note that we are abusing our notation here by referring to frames and frame morphisms rather than embeddings and commuting diagrams, but hopefully the meaning is clear.

For convenience we will identify $L$ with the lattice of $\Gamma_e$-closed subsets of $P$, and $e$ with the associated
 canonical embedding $p\mapsto \pd$. From the fact that each $f_i$ is a completely join-preserving and fixes $P$, we see that, thinking of $L$ as the
 lattice of $\Gamma_e$-closed sets, the maps $f_i$ are defined by $f_i(C)=\Gamma_{\cU_i}(C)$ for all $C\in L$. We also deduce that $\Gamma_e\leq \Gamma_{\cU_i}$ for all $i\in I$.

 By proposition \ref{P:arrow} there is a unique
 frame morphism $\phi: L\to \cI_{\bigcap_I \cU_i}$ fixing $P$.  From the definitions of the various morphisms involved
 we see that $f_i=\phi_i\circ \phi$ for all $i\in I$. Thus $\cI_{\bigcap_I\cU_i}$ and the $\phi_i$ maps induce the limit of the diagram
 associated with $\{\cI_{\cU_i}:i\in I\}$ in $\Frm_P$ as required.

To see that $F_P\circ \iota$ may not preserve terminal objects, let $P$ and $L$ be as in example \ref{E:notMod}, let
 $\JF^+$ be the category of maximal frame-generating join-specifications for $P$, and let $\cU$ be the terminal object
 of $\JF^+$, which must exist by theorem \ref{T:comps}. Let $\cU_\phi$ be as in example \ref{E:notMod}. Then, as
 $\cU_\phi$ is the largest possible join-specification for $P$ we must have $\cU\subseteq \cU_\phi$, and this inclusion must be strict as $\cU_\phi$ was shown to not be in $\JF$. So, by corollary \ref{C:contained}, there is no frame morphism from $L$ to $\cI_\cU$, and so $F_P(\cU)$ is not the terminal object of $\Frm_P$.
\item
To see that $F_P$ need not preserve finite non-empty limits we return to example \ref{E:strict}. Here $\cU_1$ and $\cU_2$ are frame-generating, but $\cU_1\cap\cU_2$ is not. So, as in the proof of theorem \ref{T:comps}(\ref{compsJF}), the limit of $\cU_1$ and $\cU_2$ is $(\cU_1\cap \cU_2)^-$, which in this case is $\{\{a,b\},\{b,c\},\{c,d\}\}$.

Consider $\cV = \cU_1^+ \cap \cU_2^+$. By theorem \ref{T:lims}(\ref{limsI}), we have $\cV \in \JF^+$, and thus
 $\eta:P\to\cI_\cV\in\Frm_P$. Moreover, by lemma \ref{L:init}, for each $i\in\{1,2\}$ there is a frame morphism $p_i:\cI_\cV \to \cI_{\cU_i}$ fixing $P$, and such that $\cI_\cV$ and the $p_i$ maps induces a cone for the relevant diagram.
However, in this case we have
\[((\cU_1\cap \cU_2)^-)^+ = (\cU_1\cap \cU_2)^-\subset \cU_1^+\cap\cU_2^+ = \cV.\]
So there is no map from $\eta:P\to\cI_\cV$ to $\eta':P\to \cI_{(\cU_1\cap\cU_2)^-}$ in $\Frm_P$, by corollary \ref{C:contained}, and thus $F_P$ does not preserve the limit corresponding to $\cU_1\cap \cU_2$.

\end{enumerate}
\end{proof}

\begin{cor}\label{C:pres}
$F_P$ and $F_P\circ\iota$ may not have left adjoints.
\end{cor}
\begin{proof}
This follows immediately from theorem \ref{T:pres} and the fact that right adjoints must preserve small limits.
\end{proof}

\section{An alternative approach}\label{S:alt}

We can also categorify the class of frame-completions for $P$ in a different way. Given $P$ and join-specifications $\cU_1$ and $\cU_2$ with
 $\cU_1\subseteq \cU_2$, the completion $\cI_{\cU_2}$ is a sublattice of $\cI_{\cU_2}$, as $\cU_2$-ideals are also $\cU_1$-ideals. Moreover, the inclusion, $i$, of $\cI_{\cU_2}$ into $\cI_{\cU_1}$ preserves arbitrary meets (intersections), including the empty meet (i.e. the top element, the ideal which is all of $P$).
If $\phi:\cI_{\cU_1}\to \cI_{\cU_2}$ is the induced frame morphism from lemma \ref{L:init}, then $\phi\circ i$ is the identity map on $\cI_{\cU_2}$, but it is not the case in general that $i\circ \phi$ is the identity on $\cI_{\cU_1}$.

Define $\Frm_{\overleftarrow{P}}$ to be the category whose objects are the same as those of $\Frm_P$, and whose arrows are order embeddings fixing $P$ (these will necessarily be completely meet-preserving - see lemma \ref{L:induced}). We can now ask similar questions about functors from $\JF$ and $\JF^+$ into $\Frm_{\overleftarrow{P}}$ as we asked for $\Frm_P$ in the preceding sections. Fortunately, we can answer all these questions very simply, as $\Frm_{\overleftarrow{P}}$ is just the opposite category to $\Frm_P$.

This duality is not new, and appears, at least implicitly in a number of places in the literature. Nevertheless, we will take a moment to articulate it precisely.

\begin{lemma}\label{L:same}
Let $\Gamma_1$ and $\Gamma_2$ be standard closure operators on $P$. Then $\Gamma_1\leq \Gamma_2$ if and only if every $\Gamma_2$-closed set is $\Gamma_1$-closed.
\end{lemma}
\begin{proof}
First suppose $\Gamma_1\leq\Gamma_2$, and let $C$ be $\Gamma_2$-closed. Then
\[C\subseteq \Gamma_1(C) \subseteq \Gamma_2(C) = C,\]
so $C$ is $\Gamma_1$-closed.

Conversely, suppose every $\Gamma_2$-closed set is $\Gamma_1$-closed, and let $S\subseteq P$. Then \[\Gamma_1(S)\subseteq \Gamma_1(\Gamma_2(S))=\Gamma_2(S),\]
so $\Gamma_1\leq\Gamma_2$.
\end{proof}

\begin{lemma}\label{L:induced}
Let $e_1:P\to L_1$ and $e_2:P\to L_2$ be join-completions, and let $g:L_2\to L_1$ be an order embedding that fixes $P$. Then:
\begin{enumerate}[(1)]
    \item If $\Gamma_1$ and $\Gamma_2$ are the closure operators associated with $e_1$ and $e_2$ respectively, and if $L_1'$ and $L_2'$ are the corresponding lattices of closed sets, then the map $g':L_2'\to L_1'$ such that the diagram in figure \ref{F:induced} commutes is given by $g'(C)=C$ for all $C\in L_2'$ (in this diagram the isomorphisms are the canonical ones).
    \item $g$ is right adjoint to the completely join-preserving map $f:L_1\to L_2$ corresponding to the map that takes $\Gamma_1$-closed sets to their $\Gamma_2$-closures (as in proposition \ref{P:arrow}).
\end{enumerate}
\end{lemma}
\begin{proof}\mbox{}
\begin{enumerate}[(1)]
\item If $y\in L_2$ then the $\Gamma_2$-closed set associated with $y$ by the canonical isomorphism is $\{p:e_2(p)\leq y\}$, and the $\Gamma_1$-closed set associated with $g(y)$ is $\{p:e_1(p)\leq g(y)\}$. But $e_1$, $e_2$ and $g$ are all order embeddings, so
\[e_2(p)\leq y \iff g(e_2(p))\leq g(y) \iff e_1(p)\leq g(y),\]
as $g$ fixes $P$. Thus $\{p:e_2(p)\leq y\}=\{p:e_1(p)\leq g(y)\}$, and the result follows.
\item By part 1 and lemma \ref{L:same} we have $\Gamma_1\leq \Gamma_2$, and so $f$ exists by proposition \ref{P:arrow}. Moreover, given $S_1\in L_1'$ and $S_2\in L_2'$, we have
\[\Gamma_1(S_1)\subseteq S_2 \iff S_1\subseteq S_2,\]
and this amounts to saying that $f$ and $g$ form an adjoint pair, with $g$ as right adjoint.
\end{enumerate}
\end{proof}

\begin{figure}
\centering\scalebox{1}{
\xymatrix{
L_1\ar@{<->}[r]_\cong & L_1' \\
L_2\ar[u]^g\ar@{<->}[r]_\cong & L_2'\ar[u]_{g'}
}}
\caption{}
\label{F:induced}
\end{figure}

\begin{prop}\label{P:op}
Let $P$ be a poset, and let $\JCop$ be the category whose objects are join-completions, $e:P\to L$, of $P$, and whose arrows are (the commuting triangles induced by) completely meet-preserving order embeddings fixing $P$.
Let $\JC$ be as in definition \ref{D:JC}. Then $\JCop=\JC^{\textup{op}}$.
\end{prop}
\begin{proof}
Both categories are thin, and their maps come in adjoint pairs. This follows from lemmas \ref{L:contained}, \ref{L:same} and \ref{L:induced}, and the uniqueness of adjoints.
\end{proof}

\begin{cor}\label{C:op}
$\Frm_{\overleftarrow{P}}=\Frm_P^{\textup{op}}$.
\end{cor}
\begin{proof}
This is immediate from proposition \ref{P:op} as $\Frm_{\overleftarrow{P}}$ and $\Frm_P$ are full subcategories of $\JCop$ and $\JC$ respectively.
\end{proof}

\section{The global picture}\label{S:global}

In the preceding sections we took a `local' view, in the sense that we fixed a poset $P$ and looked at categories of frames and join-specifications arising from it. It's natural to consider also the `global' perspective, by examining the category of all posets and suitable maps, and considering a functor from this category into $\Frm$, the category of frames and frame morphisms. The immediate obstacle here is that this doesn't make sense as it stands, as we have been thinking about join-specifications for a given poset, but not about schemes associating each poset $P$ with a join-specification $\cU_P$. Once we start thinking along these lines, another immediate problem is that if we allow $\cU_P$ to be chosen arbitrarily, the $\cU$-morphisms we want to be the maps in the category will not typically compose properly.

Fortunately, there is a substantial amount of literature addressing precisely this situation. In particular,  \cite{Ern86} provides a general framework in which questions such as these can be thoroughly answered. For reasons of brevity we will not reproduce the details here, but for convenience we will sketch out how the general theory applies in this case.

A scheme associating each poset $P$ with a join-specification $\cU_P$ is a special case of what Ern\'e calls a \emph{subset selection}, which is just a rule $\mathcal{Z}$ associating with each poset $P$ a set of its subsets $\mathcal{Z}P$. \cite{FriSch16} uses the term \emph{selection rule} for a similar purpose to us, with a similar use of ideals closed under specified joins, though in the context of partial frames.

From now on we will use $\overline{\cU}$ to refer to join-specification schemes, and, given a poset $P$, we use $\cU_P$ to refer to the join-specification for $P$ assigned by $\overline{\cU}$.

\begin{defn}[$\PosU$, $\PosIU$]
Given a join-specification scheme $\overline{\cU}$, define $\PosU$ to be the class of all posets, along with the class of all $\cU_P$-morphisms taken across all posets $P$. We define $\PosIU$ similarly, except that we restrict to the injective $\cU_P$-morphisms.
\end{defn}

We are interested in schemes $\overline{\cU}$ such that $\PosU$ and $\PosIU$ are categories, which happens if and only if the $\cU_P$-morphisms all compose appropriately. We will call such join-specification schemes \emph{categorical}. Note that Ern\'e uses the term \emph{compositive} for a similar purpose. We use a different word to reinforce that we are specifically talking about the $\cU_P$-morphisms, and not the most obvious class of maps that would be associated with a subset selection in Ern\'e's system.

What we call a $\cU_P$-morphism is what \cite{Ern86} would call a \emph{$\mathcal{Z}P$-join preserving map}. It follows from \cite[propositions 1.4 and 1.8]{Ern86} that a join-specification scheme $\overline{\cU}$ will be categorical if
 and only if all the maps in the category are \emph{continuous}, where a map $f:P\to Q$ is continuous if whenever
 $C\subseteq Q$ is a $\cU_Q$-ideal, $f^{-1}[C]$ is a $\cU_P$-ideal. Note that this use of continuity, though
 deliberately suggestive of topology, is not equivalent, as the set of $\cU_P$-ideals of some poset $P$ need not be
 closed under finite unions, even if this $\cU_P$ is assigned by a categorical scheme.

The term \emph{subset system} has also appeared in the literature \cite{BanNel82,WWT78,Mes83,Ern86}, meaning a subset selection $\mathcal{Z}$ with the strong property that whenever $f:P\to Q$ is monotone and $X\in \mathcal{Z}P$, we must
 have $f[X]\in \mathcal{Z}Q$. By \cite[corollary 1.9]{Ern86}, if a join-specification scheme $\overline{\cU}$ is also a subset system, then it will be categorical.

Given a join-specification scheme $\overline{\cU}$, and a $\cU_P$-morphism $f:P\to Q$, we can define a map $f^+: \cI_{\cU_P}\to \cI_{\cU_Q}$ by
\[f^+(C) = \Gamma_{\cU_Q}(f[C]).\]
When $\overline{\cU}$ is categorical, this map preserves arbitrary joins. For a proof of this claim see e.g. \cite[section
 2]{Ern86}, or just argue directly from the definitions using the fact that if $\overline{\cU}$ is categorical then $f$ is continuous, in the sense defined above. Thus, if a join-specification scheme $\overline{\cU}$ is categorical, then it also defines a functor $F_\cU$ from
 $\PosU$ to the category $\LatR$ of complete lattices and completely join-preserving (i.e. residuated) maps. $F_\cU$ takes a poset $P$ to $\cI_{\cU_P}$,
 and a map $f:P\to Q$ to $f^+$. This is essentially a special case of \cite[proposition 2.1]{Ern86}. In the case where $\cU_P$ is a frame for every $P$, note the following lemma and its corollary.

\begin{lemma}\label{L:lift}
Let $\overline{\cU}$ be a categorical join-specification scheme, and let $f:P\to Q$ be an injective $\cU_P$-morphism. Suppose $\cI_{\cU_P}$ and $\cI_{\cU_Q}$ are both frames. Then $f^+:\cI_{\cU_P}\to \cI_{\cU_Q}$ is a frame morphism.
\end{lemma}
\begin{proof}
Since, as explained above, it is always the case that $f^+$ preserves arbitrary joins, we need only show it preserves binary meets. So let $C_1,C_2\in \cI_{\cU_P}$. Then, using corollary \ref{C:arrow} we have
\begin{align*}f^+(C_1\wedge C_2) &= f^+(C_1\cap C_2) \\
&= \Gamma_{\cU_Q}(f[C_1\cap C_2)])\\
&= \Gamma_{\cU_Q}(f[C_1]\cap f[C_2])\\
&= \Gamma_{\cU_Q}(f[C_1])\cap \Gamma_{\cU_Q}(f[(C_2)])\\
&= f^+(C_1)\wedge f^+(C_2).
\end{align*}
This shows that $f^+$ preserves binary meets as required.
\end{proof}

Lemma \ref{L:lift} is a kind generalization of corollary \ref{C:arrow}, as we can let $f$ be an order isomorphism.  We thus have:

\begin{cor}\label{C:func}
Let $\overline{\cU}$ be a categorical join-specification scheme such that $\cI_{\cU_P}$ is a frame for every poset $P$. Then $F_\cU$ restricts to a functor $F^i_\cU$ from $\PosIU$ to $\Frm$.
\end{cor}
\begin{proof}
This follows directly from lemma \ref{L:lift}.
\end{proof}

We could, of course, refine corollary \ref{C:func} by relaxing the requirement that $\cI_{\cU_P}$ is a frame for every poset $P$, and instead restricting the domain of $F^i_\cU$ appropriately.

Example \ref{E:emNec} below shows that the restriction to injective maps is necessary if we want the lift $f^+$ to preserve binary meets in general. Example \ref{E:notInj} shows that the lift $f^+$ may not be 1-1, even if $f$ is an order embedding.

\begin{ex}\label{E:emNec}
Let $P$ be the two element antichain $\{a,b\}$, and let $Q$ be the two element chain with carrier $\{c,d\}$ and such
 that $c\leq d$. Define $\overline{\cU}$ to be the scheme that assigns to a poset the set of all its subsets whose join is
 defined, so $\cU_P =\{\{a\},\{b\}\}$, and $\cU_Q=\{\emptyset, \{c\},\{d\}, \{c,d\}\}$. Then the $\overline{\cU}$-morphisms are the
 completely join-preserving monotone maps, and so $\overline{\cU}$ is categorical. Define $f:P\to Q$ by $f(a)=f(b)=d$. It's easy to see that $f$ is a $\overline{\cU}$-morphism.

Consulting the table in figure \ref{F:empty}, we see that $\cI_{\cU_P} = \{\emptyset, \{a\}, \{b\}, \{a,b\}\}$, and
 $\cI_{\cU_Q}=\{\{c\}, \{c,d\}\}$ (both considered as lattices ordered by inclusion). Both $\cI_{\cU_P}$ and
 $\cI_{\cU_Q}$ are obviously frames. Consider the lift $f^+:\cI_{\cU_P}\to \cI_{\cU_Q}$. We have
\[f^+(\{a\}\cap\{b\}) = f^+(\emptyset) = \Upsilon_{\cU_Q}(\emptyset)=\{c\}\neq \{c,d\} = f^+(\{a\})\cap f^+(\{b\}),\]
and thus $f^+$ does not preserve binary meets.
\end{ex}

\begin{ex}\label{E:notInj}
Let $\overline{\cU}$ again be the scheme that assigns to a poset the set of all its subsets whose join is defined. Let $P$ be the
 antichain $\{a,b,c\}$, and let $Q$ be the poset with carrier $\{a',b',c',\top\}$, where $\{a',b',c'\}$ is an antichain
 and $\top$ is the top element. Then  $\cI_{\cU_P} = \wp(P)$, and $\cI_{\cU_Q}=\{\emptyset, \{a'\}, \{b'\}, \{c'\},
 \{a',b',c',\top\} \}$. Define $f:P\to Q$ by $f(x)=x'$ for $x\in \{a,b,c\}$. Then $f$ is a $\cU_P$-embedding, but
 $f^+(\{a,b\})= \{a',b',c',\top\} = f^+(\{b,c\})$, so $f^+$ is not 1-1.
\end{ex}

If $\overline{\cU}$ is a categorical join-specification scheme than it follows from \cite[theorem 2]{Schm74} that the inclusion of the category of $\LatR$ into $\PosU$ has the functor $F_\cU$ as a left adjoint (see also \cite[proposition 2.3]{Ern86} for a generalization, and \cite[theorem 6.2]{ErnZha01} for a related categorical equivalence).

Motivated by corollary \ref{C:func}, and inspired by the adage that `adjoint functors arise everywhere', a natural question is, for which choices of $\overline{\cU}$, and for which subcategories of $\PosU$ and $\Frm$, is the restriction of $F_\cU$ a left adjoint to the corresponding restriction of the inclusion functor from $\Frm$ back the other way?

We will not address this question in any depth, but we note that examples \ref{E:emNec} and \ref{E:notInj} reveal a
 potential obstacle. Example \ref{E:emNec} shows that if we don't restrict to at least $\PosIU$ then the appropriate restriction of $F_\cU$ may not be functorial. At the same time, example \ref{E:notInj} demonstrates that even lifts of
 embeddings may not be 1-1, so the `inclusion functor' is not always even defined, even if we do restrict to a subcategory of $\PosIU$.

This is not quite the end of the story however, as $\cI_{\cU_Q}$ from example \ref{E:notInj} is not a frame, so the poset
 $Q$ would have be excluded from our hypothetical subcategory of $\PosU$ anyway. This turns out to not be an accident, and we can end with a modest positive result.

\begin{lemma}\label{L:embed}
Let $\overline{\cU}$ be the join-specification scheme that assigns to every poset the set of all its subsets whose
 joins are defined. Let $f:P\to Q$ be an order embedding, and suppose that $\cU_P$ and $\cU_Q$ are both frame
 generating. Then the lift $f^+$ is an order embedding.
\end{lemma}
\begin{proof}
Let $C,D\in\cI_{\cU_P}$ and suppose $f^+(C)\subseteq f^+(D)$. Then, by definition of $f^+$, we have $\Upsilon_{\cU_P}(f[C]) \subseteq \Upsilon_{\cU_P}(f[D])$. So, in particular, $f[C]\subseteq \Upsilon_{\cU_P}(f[D])$. Since $\cU_Q$ is frame-generating, and since $f$ is an order embedding,
 this means that given $f(c)\in f[C]$ there is $T_c\subseteq D$ with $\bv f[T_c] = f(c)$ (appealing once again to theorem \ref{T:gen}(\ref{D})). It follows that $\bv T_c = c$. This is true for all $c\in C$, so $C\subseteq \Upsilon_{\cU_P}(D) = D$. Since $f^+$ is always monotone we are done.
\end{proof}

\begin{cor}
Let $\overline{\cU}$ be as in lemma \ref{L:embed}, let $\Pos^F_\cU$ be the category whose objects are those posets for which $\cU_P$ is frame-generating, and whose
 arrows are the $\cU_P$-embeddings. Let $F$ be the restriction of $F_\cU$ to $\Pos^F_\cU$, and define $\Frm_e$ to be the
 category of frames and frame embeddings. Let $U$ be the inclusion functor of $\Frm_e$ into $\Pos_F^\cU$. Then $F \dashv U$.
\end{cor}
\begin{proof}
Lemmas \ref{L:lift} and \ref{L:embed} ensure that $F$ is a functor. For $P\in \Pos^F_\cU$ we have $U(F(P)) =
 \cI_{\cU_P}$, and for $L \in \Frm_e$ we have $F(U(L)) \cong L$ via the map taking $C \in \cI_{\cU_L}$ to $\bv C$.
 Define the unit of the adjunction by $\eta_P:p \mapsto \pd$ for $p\in P$, and define the counit by
 $\varepsilon_L: C \mapsto \bv C$. We can appeal to the standard adjunction condition articulated as \cite[corollary 2.2.6]{Lein14}. Checking that $\eta$ and $\varepsilon$ define natural transformations, and that the appropriate triangle identities hold, is routine.
\end{proof}

\bibliographystyle{plain}

\end{document}